\documentclass{article}
\usepackage[utf8]{inputenc}
\usepackage[margin= 1.75 cm]{geometry}

\usepackage{amsthm}
\usepackage{enumitem}
\usepackage{amsmath}
\usepackage{amssymb}
\usepackage{setspace}
\usepackage{mathtools}
\usepackage{verbatim}
\usepackage{csquotes}
\usepackage{graphicx}
\usepackage[hidelinks]{hyperref}
\usepackage{thm-restate}
\usepackage{cleveref}
\usepackage{graphicx}
\usepackage{mathrsfs}
\usepackage{enumitem}
\usepackage{framed}
\usepackage{subcaption}
\usepackage{crossreftools}
\usepackage{floatrow}
\usepackage[T1]{fontenc}
\theoremstyle{plain}

\newtheorem*{thm*}{Theorem}
\newtheorem{theorem}{Theorem}
\Crefname{theorem}{Theorem}{Theorems}

\newtheorem*{lem*}{Lemma}
\newtheorem{lemma}[theorem]{Lemma}
\Crefname{lemma}{Lemma}{Lemmas}

\newtheorem*{claim*}{Claim}

\crefname{claim}{Claim}{Claims}
\Crefname{claim}{Claim}{Claims}

\newtheorem{prop}[theorem]{Proposition}
\Crefname{prop}{Proposition}{Propositions}

\newtheorem{corollary}[theorem]{Corollary}
\crefname{corollary}{Corollary}{Corollaries}

\newtheorem{conj}[theorem]{Conjecture}
\crefname{conj}{Conjecture}{Conjectures}

\newtheorem*{conj*}{Conjecture}

\Crefname{qn}{Question}{Questions}

\Crefname{obs}{Observation}{Observations}

\Crefname{ex}{Example}{Examples}

\theoremstyle{definition}

\Crefname{prob}{Problem}{Problems}

\newtheorem{defn}[theorem]{Definition}
\Crefname{defn}{Definition}{Definitions}

\newtheorem*{defn*}{Definition}

\theoremstyle{remark}

\renewenvironment{proof}[1][]{\begin{trivlist}
\item[\hspace{\labelsep}{\bf\noindent Proof#1.\/}] }{\qed\end{trivlist}}

\newcommand{\ceil}[1]{
    \left\lceil #1 \right\rceil
}
\newcommand{\floor}[1]{
    \left\lfloor #1 \right\rfloor
}

\newcommand{\eps}{\varepsilon}

\renewcommand{\P}{\mathbb{P}}
\newcommand{\E}{\mathbb{E}}
\newcommand{\N}{\mathbb{N}}
\newcommand{\R}{\mathbb{R}}

\newcommand{\G}{\mathcal{G}}

\DeclareMathOperator{\bin}{Bin}

\expandafter\def\expandafter\normalsize\expandafter{%
    \normalsize
    \setlength\abovedisplayskip{8pt}
    \setlength\belowdisplayskip{8pt}
    \setlength\abovedisplayshortskip{4pt}
    \setlength\belowdisplayshortskip{4pt}
}

\usepackage[square,sort,comma,numbers]{natbib}
 \setlist[itemize]{leftmargin=*}

\makeatletter
\newcommand{\optionaldesc}[2]{%
  \phantomsection
  #1\protected@edef\@currentlabel{#1}\label{#2}%
}
\makeatother

\title{Towards the Erd\H{o}s-Gallai Cycle Decomposition Conjecture}
\author{
Matija Buci\'c\thanks{School of Mathematics, Institute for Advanced Study and Department of Mathematics, Princeton University, Princeton, USA. Email: \href{mailto:matija.bucic@ias.edu} {\nolinkurl{matija.bucic@ias.edu}}. Supported in part by NSF Grant CCF-1900460.} 
\and
Richard Montgomery\thanks{Mathematics Institute, University of Warwick, Coventry, CV4 7AL, UK.
Email: \href{mailto:richard.montgomery@warwick.ac.uk} {\nolinkurl{richard.montgomery@warwick.ac.uk}}.
Supported by the European Research Council (ERC) under the European Union Horizon 2020 research and innovation
programme (grant agreement No. 947978) and the Leverhulme trust.}
}
\date{}

\begin{document}

\maketitle

\begin{abstract}
    In the 1960's, Erd\H{o}s and Gallai conjectured that the edges of any $n$-vertex graph can be decomposed into $O(n)$ cycles and edges. We improve upon the previous best bound of $O(n \log \log n)$ cycles and edges due to Conlon, Fox and Sudakov, by showing an $n$-vertex graph can always be decomposed into $O(n \log^{\star} n)$ cycles and edges, where $\log^{\star}n$ is the iterated logarithm function. 
\end{abstract}

\section{Introduction}
When is it possible to decompose a graph into edge disjoint subgraphs with certain properties? 
Many classical problems in extremal combinatorics fall within this framework and its natural hypergraph generalisation, while decomposition problems have strong links to many other fields, including the design of experiments, coding theory, complexity theory and distributed computing (see, for example, \cite{network_decompositions1989,designs_applications1989,jukna_complexity2006}). 
The particular case where we seek to decompose a graph into cycles has a long history, dating back to the 18th century and Euler's result on the existence of Euler tours. As Veblen~\cite{veblen1,veblen2} observed for his algebraic approach to the Four-Colour Theorem, Euler's result immediately implies that any graph with even vertex degrees (i.e., any \emph{Eulerian} graph) has a decomposition into cycles. As it is immediate that any graph with a vertex of odd degree cannot be decomposed into cycles, this exactly characterises which graphs have cycle decompositions.

Another very classical cycle decomposition result is due to Walecki \cite{lucas1883recreations} from 1892, who showed it is possible to decompose any complete graph with an odd number of vertices into Hamilton cycles. This gives a cycle decomposition into few cycles, indeed, into optimally few cycles. This raises a very natural question of whether every Eulerian graph has a cycle decomposition into few cycles? That only $O(n)$ cycles might be needed to decompose any $n$-vertex Eulerian graph is easily seen to be equivalent to the following classical conjecture of Erd\H{o}s and Gallai \cite{erdos1966representation} dating back to the 1960's, which is one of the major open problems on graph decompositions.

\begin{conj}\label{conj:EG}
Any $n$-vertex graph can be decomposed into $O(n)$ cycles and edges.
\end{conj}

While \Cref{conj:EG} is equivalent to conjecturing that every $n$-vertex Eulerian graph can be decomposed into $O(n)$ cycles, as noted above, if they both hold then the optimal implicit constants in these conjectures seem likely to be different. For the Eulerian problem, Haj\'os conjectured that $\frac n2$ cycles should be sufficient \cite{lovasz1968covering} (see also~\cite{dean1986smallest,chung1980,bondyEG1990,fan2003covers,girao2021path}), while the best known lower bound for the number of cycles and edges required in \Cref{conj:EG} is $(\frac{3}{2}-o(1))n$, as observed by Erd\H{o}s in 1983 \cite{erdHos1983some}, improving on a previous construction of Gallai \cite{erdos1966representation} (see \Cref{sec:final}).

Since its formulation, the Erd\H{o}s-Gallai Conjecture has often been highlighted (see, for example,~\cite{bondyEG1990,pyber1996covering,pyberEG1991,dano2015,glock2016optimal,girao2021path,conlon2014cycle}), with Erd\H{o}s himself mentioning it in many of his open problem collections \cite{erdHos1983some,erdHossome1971,erdossolved1981,erdos1973problems}.
Despite this attention, and a lot of work on related problems over the years, direct progress towards the Erd\H{o}s-Gallai Conjecture has only been made within the last decade. The previous related results, which we discuss first, are mostly on the analogous path decomposition problem and the covering version of the Erd\H{o}s-Gallai conjecture.

\textbf{Path decompositions.} In the 1960's, Gallai~\cite{lovasz1968covering} posed the analogous path decomposition version of \Cref{conj:EG}. In particular, he conjectured that any connected $n$-vertex graph can be decomposed into at most $\frac{n+1}{2}$ paths. Lov\'asz \cite{lovasz1968covering} in 1968 proved that any graph can be decomposed into at most $n-1$ paths. This follows easily from his complete solution to the problem of how many paths \emph{or} cycles one needs to decompose an $n$-vertex graph, to which the answer is $\floor{\frac n2}$. Currently, the best general bound in the path decomposition problem is due independently to Dean and Kouider~\cite{dean2000gallai} and Yan~\cite{yan1998path}, who showed that any graph can be decomposed into at most $\lfloor \frac23 n\rfloor$ paths.
Gallai's path decomposition conjecture is known to hold for quite a few special classes of graphs, with connected planar graphs being the most recent addition to the list. This latest result is due to Blanch\'e, Bonamy and Bonichon in \cite{blanche2021gallai}, where a more exhaustive list of partial results can also be found. 

\textbf{Covering problems.} Another interesting direction which has attracted a lot of attention is the covering version of \Cref{conj:EG}, in which we do not insist that the cycles we find should be disjoint, only that together they contain all the edges of the host graph. 
In 1985, Pyber~\cite{pyber1985erdHos} proved the covering version of the Erd\H{o}s-Gallai conjecture, showing that the edges of any $n$-vertex graph can be covered with $n-1$ cycles and edges. The analogous covering version of Gallai's conjecture, raised by Chung \cite{chung1980} in 1980, has been settled first approximately by Pyber~\cite{pyber1996covering} in 1996 and then completely by Fan~\cite{fan2002subgraph} in 2002, who showed that the edges of any connected graph can be covered by $\ceil{ \frac{n}{2}}$ paths. The covering version of Haj\'os's conjecture was also solved by Fan~\cite{fan2003covers}, who showed that any $n$-vertex Eulerian graph can be covered by at most $\lfloor \frac{n-1}{2}\rfloor$ cycles, settling another conjecture of Chung. As with the other two covering results above, this bound is best possible.

\textbf{Results on the Erd\H{o}s-Gallai conjecture.} In more recent years, the Erd\H{o}s-Gallai conjecture (along with more accurate results on the implicit bounds) has been shown to hold for two large specific classes of graphs -- random graphs and graphs with linear minimum degree.
The conjecture was first established for a typical binomial random graph $G(n,p)$ (for any $p=p(n)$) by Conlon, Fox and Sudakov \cite{conlon2014cycle}. Kor{\'a}ndi, Krivelevich, and Sudakov \cite{dano2015} found the correct leading constant here, showing that $(\frac{1}{4}+\frac{p}{2}+o(1))n$ cycles and edges are typically sufficient to decompose $G(n,p)$. For constant edge probability $p$, Glock, K{\"u}hn, and Osthus~\cite{glock2016optimal} were even able to determine with high probability the exact minimum number of cycles and edges required to decompose a (quasi)random graph. On the other hand, the Erd\H{o}s-Gallai conjecture was first shown to hold for graphs with linear minimum degree again by Conlon, Fox and Sudakov \cite{conlon2014cycle}. Very recently, the asymptotically correct bound of $(\frac32+o(1)) n$ cycles and edges has been proved by Gir{\~a}o, Granet, K{\"u}hn, and Osthus~\cite{girao2021path} for large graphs with linear minimum degree.

A fundamental challenge towards establishing the Erd\H{o}s-Gallai conjecture is its generality, and indeed these previous results make progress only by imposing a fairly strong constraint on the structure or randomness of the graph. 
For almost 50 years, the best known bound in the general case of the Erd\H{o}s-Gallai conjecture (as observed by Erd\H{o}s and Gallai) came from a simple argument involving the iterative removal of a longest cycle, which shows that an $n$-vertex graph can always be decomposed into $O(n\log n)$ cycles and edges. In 2014, Fox, Conlon and Sudakov~\cite{conlon2014cycle} made the first major breakthrough on this problem, showing that such a decomposition with only $O(n\log\log n)$ cycles and edges always exists. Here we will give the following improvement on this bound, where $\log^\star n$ is the iterated logarithm function.

\begin{restatable}{theorem}{logstar}
\label{thm:logstar}
Any $n$-vertex graph can be decomposed into $O(n\log^\star n)$ cycles and edges.
\end{restatable}

Key to the decompositions used by Conlon, Fox and Sudakov~\cite{conlon2014cycle} was to show that \emph{a)} graphs $H$ with certain expansion properties can be decomposed into $O(|H|)$ cycles and few edges and \emph{b)} any $n$-vertex graph $G$ can be decomposed nicely into such `expanders' $H$ and a small number of leftover edges. Combined, this gives a decomposition of $G$ into $O(n)$ cycles and some leftover edges, and it can be shown that iterating this on the leftover edges while removing any particularly long cycles causes the average degree of the leftover edges to drop significantly each time, so that after $\log\log n$ iterations the decomposition given by~\cite{conlon2014cycle} is achieved.

To prove \Cref{thm:logstar}, essentially, we need the average degree of the leftover edges to drop much faster, and so at \emph{b)} we have to take a much weaker condition on the `expanders' $H$. Effectively we replace the strong expansion used in \cite{conlon2014cycle}, with a very weak \emph{sublinear} expansion (as introduced by Koml\'os and Szemer\'edi~\cite{K-Sz-1,K-Sz-2}), in particular using a \emph{robust} sublinear expansion where sets expand sublinearly despite the additional removal of a possibly-superlinear set of edges (see \Cref{sec:generalexpansionchat} for a discussion of these forms of expansion and their background). 

Using this much weaker form of expansion introduces a raft of issues when we decompose an expander into few cycles and edges (for \emph{a)} above), resulting in a very different approach to that used in \cite{conlon2014cycle}. In order to do this, we introduce a range of new tools, which we hope will find further applications. In particular, we would highlight a new approach to robust sublinear expansion (see \Cref{sec:expansion}) and the (surprisingly difficult) result that randomly sampling the vertices of an expander is likely to induce a subgraph with a (somewhat weaker) expansion property (see \Cref{lem:expandintorandom2}). Additional new tools include a similar result but while randomly sampling edges, the (almost) decomposition of any graph into robust sublinear expanders, and the finding of a sparse `connective skeleton' in expanders
to connect vertex pairs with paths.
These, and other tools, and how they come together to prove Theorem~\ref{thm:logstar}, are discussed in \Cref{subsec:sketch}.

As discussed in \Cref{sec:generalexpansionchat}, sublinear expansion has been useful in many different settings in which our tools may also be useful (see \Cref{sec:final} for some examples). In particular, \emph{robust} sublinear expansion (specifically considering the deletion of superlinearly many edges) is a very recent concept and we hope our new perspective and tools will contribute to its development and use. 
Several of our intermediate results and tools might also ultimately prove useful towards proving the Erd\H{o}s-Gallai conjecture in full as they often decompose any $n$-vertex graph into $O(n)$ cycles and a graph with some other structure imposed.
This is discussed further in our concluding remarks in \Cref{sec:final}.


\section{Preliminaries}\label{sec:prelim}

After we introduce our notation, we give a detailed sketch of our methods before outlining the rest of the paper.

\subsection{Notation}
Given a graph $G$ we will denote by $V(G)$ and $E(G)$ its vertex and edge set, respectively. Given a vertex $v\in V(G)$, we denote its degree by $d_G(v)$ and the set of its neighbours by $N_G(v)$. We write $\Delta(G)$ for the maximum degree of a vertex in a graph $G$. Given a subset of vertices $U \subseteq V(G)$ we denote by $N_G(U)$ the set of vertices in $V(G) \setminus U$ which have a neighbour in $U$. 
Given $U \subseteq V(G)$ we define $B^{i}_G(U)$ as the set of vertices at distance at most $i$ from a vertex of $U$ in the graph $G$, i.e.\ the ball of radius $i$ around $U$ in $G$, and write simply $B_G(U)=B^1_G(U)$. Given $V \subseteq V(G)$ we write $G[V]$ for the subgraph of $G$ induced by the vertex set $V$, and write $G \setminus V$ for $G[V(G) \setminus V]$. Given $F \subseteq E(G)$ we write $G-F$ for the subgraph of $G$ obtained by deleting all the edges in $F$. Given multiple (hyper)graphs $H_1,\ldots, H_t$ we write $H_1 \cup \ldots \cup H_t$ for the (hyper)graph with vertex set $\bigcup_{i\in [t]}V(H_i)$ and edge set $\bigcup_{i\in [t]}E(H_i)$. Given vertices $v$ and $u$, by a $vu$-path/walk we refer to a path/walk joining $v$ and $u$. 

We write $X \sim \text{Bin}(n,p)$ to mean that $X$ is a random variable distributed according to the binomial distribution with parameters $n$ and $p$. We denote by $\G(n,p)$ the binomial random graph defined as the graph with vertex set $[n]$ in which we sample every edge with probability $p$ independently from all other edges. We write $G \sim \G(n,p)$ to mean that $G$ is sampled according to $\G(n,p).$

For each $n\geq 1$, $[n]=\{1,\ldots,n\}$. All our logarithms have base two. For each $k\geq 1$, let $\log^{[k]}(n)=\underbrace{\log \log \ldots \log}_{k \text{ times}} n$, and let $\log^{[0]}n=n$. 
The iterated logarithm function $\log^{\star} n$ is the minimum number of times we need to apply the logarithm function to $n$ until it becomes at most one, that is, the least $k\geq 0$ such that $\log^{[k]}n\leq 1$.

Throughout the paper, we make no attempt to optimise constants and logarithmic factors; often, we are wasteful to improve readability. With the same goal, we also omit floor and ceiling signs wherever they are not crucial.

\subsection{Proof sketch}\label{subsec:sketch}

Our methods to find an (edge) decomposition into cycles and edges is iterative, where a single iteration, applied to an $n$-vertex graph $G$ with average degree $d$, performs the following steps for some appropriately large constant $C$. 
\begin{itemize}
    \item Repeatedly remove any cycle of length at least $d$ and add it to the decomposition, giving at most $\frac n2$ new cycles.
    \item Decompose $G$ into edge disjoint subgraphs $R_i$ with a certain expansion property which, combined with the lack of long cycles, guarantees that $|R_i|= O(d \log^4 d)$. These subgraphs are almost vertex disjoint, so that $\sum_{i}|R_i|\leq 2n$.
    \item Decompose each of the subgraphs $R_i$ into cycles and edges, using in total $O(n)$ cycles and $O(n\log^Cd)$ edges.
\end{itemize}
This finds a decomposition of $G$ into $O(n)$ cycles and a subgraph consisting of leftover edges of average degree at most $O(\log^Cd)$. We now iterate by applying the same argument to our much sparser graph consisting of leftover edges. After at most $O(\log^\star d)$ iterations we will be left with a graph of constant average degree. We then simply make all its edges part of our decomposition, which together with the $O(n)$ cycles we found at each of the $O(\log^\star d)$ iterations gives our desired decomposition for \Cref{thm:logstar} -- in fact, using only $O(n\log^\star d)$ edges for any $n$-vertex graph with average degree $d$. 

The majority of our work lies in carrying out the final step in this iteration, where the key part of this step is to show the following intermediate result. 

\begin{restatable}{theorem}{decompcyclelinear}
\label{thm:decompexander} There exists $C>0$ such that any $r$-vertex graph decomposes into $O(r)$ cycles and $O(r\log^{C}r)$ edges.
\end{restatable}

Note that this intermediate result is applied to graphs $R_i$, each of which has order at most $O(d \log^4 d)$, so that, as $\sum_{i}|R_i|\leq 2n$, we get in total $O(n)$ cycles and $O(\sum_{i}|R_i|\log^C|R_i|)=O(\sum_{i}|R_i|\log^Cd)=O(n\log^Cd)$ edges in total  from the decomposition in the final step of the iteration.

Though we use a slightly modified iteration argument, till now our approach has the same structure as the one taken by Conlon, Fox and Sudakov in \cite{conlon2014cycle}, where they prove a weaker version of \Cref{thm:decompexander} in which they allow, instead of $O(r\log^C r)$, up to $O(r^{2-1/10})$ edges in the decomposition. This much weaker bound leads them to iterate $O(\log\log d)$ times and thus use a decomposition using $O(n\log\log d)$ cycles and edges in total. The key difference is that we are able to replace the extremely strong expansion properties used in~\cite{conlon2014cycle} with a very weak form of expansion. To make this change successfully, we need to carefully develop the weak expansion property we use (which originates with Koml\'os and Szemer\'edi~\cite{K-Sz-1,K-Sz-2}) as well as solve the variety of problems caused by working with this weak expansion. For this development, the main insight is the new perspective we bring to robust sublinear expansion. That we can decompose an arbitrary graph into edges and expanders follows relatively naturally from the definition of robust sublinear expansion (see \Cref{sec:decomp-into-expanders} for more details on the expander partitioning lemma), and the same decomposition result allows us to carry out the second step of the iteration mentioned above. This leaves the difficult task of decomposing an expander graph into few cycles and edges, which we now discuss. 

\smallskip

\noindent\textbf{Decomposing expanders into few cycles and edges.}
For some constant $C>0$, we now assume that we wish to decompose an $n$-vertex graph $G$ with the following (slightly simplified) expansion condition (see \Cref{sec:expansion} for the full condition we use): for each $U\subseteq V(G)$ and $F\subseteq E(G)$ with $|U|\leq \frac23 n$ and $|F|\leq |U|\log^Cn$, we have
\begin{equation}\label{expandprop}
|N_{G-F}(U)|\geq \frac{1}{\log^2n}|U|.
\end{equation}

As the size of the neighbourhood guaranteed in \eqref{expandprop} is smaller than $|U|$, this type of expansion is known as \emph{sublinear expansion}. 

Roughly speaking, our main strategy is to set aside a `sparse connecting skeleton' $H\subseteq G$, before initially decomposing the edges of $G-H$ into $O(n)$ paths and cycles using a result of Lov\'asz~\cite{lovasz1968covering} stated in the introduction. We then use short paths from the connecting skeleton $H$ to connect up each initial path into a cycle, before simply taking each unused edge of $H$ as part of our decomposition. In order for this to produce a correct decomposition we need $H$ to be very sparse, in particular with at most $n\log^{O(1)}n$ edges.

To aid the connection of the paths from the initial path/cycle decomposition, we need that each vertex does not appear too often as an endvertex of these paths. This we ensure by proving a simple, but crucial, corollary of Lov\'asz's result (see \Cref{cor:lovasz}), which will allow us to decompose $G-H$ into a collection $\mathcal{P}$ of $O(n)$ paths in which each vertex appears as an endvertex at most twice, so that the endvertices of the paths are well spread across the graph.

More problematically, note that in order to get an actual cycle we need to connect the endvertices of each path $P\in \mathcal{P}$ using a path in $H$ which is internally vertex disjoint from $P$. To deal with this, we change this outline slightly as follows. We partition $V(G)=V_1\cup V_2\cup V_3$ by placing each vertex independently into a set $V_i$ uniformly at random, and show that $G$ contains sparse subgraphs $H_1,H_2,H_3$ (each with $n\log^{O(1)} n$ edges) with the following property for each $i\in [3]$, where a \emph{path through $V_i$} is one whose interior vertices are all in $V_i$. We place no restriction on the endvertices themselves, so in particular a single edge path is a path through any set since it contains no interior vertices.
\begin{enumerate}[label = \textbf{P}]
\item For any set $\mathcal{P}\subseteq \binom{V(G)}{2}$ such that each vertex appears in at most $2$ pairs in $\mathcal{P}$, there are edge disjoint paths $P_{xy}$, $\{x,y\}\in \mathcal{P}$, such that, for each $\{x,y\}\in \mathcal{P}$, $P_{xy}$ is an $xy$-path through $V_i$ in $H_i$ with length $O(\log^7n)$.\label{prop:sketch}
\end{enumerate}

We then split the edges of $G-H_1-H_2-H_3$ into three subgraphs $G_1,G_2,G_3$ in such a way that $V(G_i)= V_{i+1} \cup V_{i+2},$ for each $i\in [3]$, with indices taken modulo $3$. Applying our path decomposition corollary to each $G_i$ then gives a decomposition of all the edges outside of $H_1 \cup H_2 \cup H_3$ into paths whose endvertices are well spread across the graph. Note that the paths decomposing $G_i$ completely avoid $V_i$ so by using the property~\ref{prop:sketch} we can connect each of these paths into actual cycles using edges of $H_i$. In total we find $O(n)$ edge disjoint cycles which use all the edges of $G-H_1-H_2-H_3$ so the total number of uncovered edges, which all belong to $H_1\cup H_2\cup H_3$, is small. 

With such a weak expansion property as that at \eqref{expandprop}, whether we can do this is initially far from clear. Building up to this, we ask the following three questions. 

\vspace{-0.2cm}

\begin{enumerate}[label = \roman{enumi})]
    \item  Can we connect pairs of vertices with edge disjoint paths using the whole of $G$? \\I.e., does property \ref{prop:sketch} hold if $H_i=G$ and $V_i=V(G)$?
    \item If so, can we do this using only a random subset of vertices $V_i$ for the interior vertices of the paths? \\I.e., does property \ref{prop:sketch} hold if $H_i=G$?
    \item If so, can we do this using only a sparse subgraph $H_i$ of $G$? \\I.e., can we find a sparse $H_i \subseteq G$ so that property \ref{prop:sketch} holds for $H_i$?
\end{enumerate}

\vspace{-0.2cm}

\smallskip \textbf{i) Finding edge disjoint paths in $G$.} The expansion condition on $G$ at \eqref{expandprop} is sufficient to imply that any pair of vertices in $H$ are connected by a path of length $O(\log^3n)$ by expanding the neighbourhoods around each of the two vertices until they become large enough that they must overlap. Moreover, the robustness of our condition at \eqref{expandprop} (i.e., that this expansion can avoid using an arbitrary, but not too large, set of edges $F$) allows us, with only a bit more work, to show that, for any collection $\mathcal{P}$ of pairs of vertices as in property \ref{prop:sketch}, we could find at least $\Theta(\log^4n)\cdot |\mathcal{P}|$ edge disjoint paths in $G$ which each connect \emph{some} vertex pair in $\mathcal{P}$ and have length $O(\log^3n)$. As this holds in fact for any subset $\mathcal{P}'\subseteq \mathcal{P}$ in place of $\mathcal{P}$, this allows us to use the Aharoni-Haxell hypergraph matching theorem (see \Cref{thm:hyperhall}) to select, for each pair $\{x,y\}\in \mathcal{P}$, an $xy$-path in $G$, so that all these paths are edge disjoint. In total this allows us to answer question i) positively.

\textbf{ii) Connecting through the random vertex subset $V_i$.}  With $i\in [3]$ and $V_i$ a random subset of $V$ with size approximately $\frac{n}3$ as chosen above, unfortunately it seems it does not follow easily that $H[V_i]$ likely satisfies a similar expansion property to~\eqref{expandprop}. Indeed, firstly, if the inequality \eqref{expandprop} is tight for a set $U$, then with probability $\exp(-\Theta(|U|/\log^2n))$ we have that $U$ has no neighbours selected into our random subset $V_i$, too high a probability to naively take a union bound to avoid this event over all sets $U$ with any fixed size. Secondly, the robustness condition we need in order to find multiple edge disjoint paths (to then apply the Aharoni-Haxell hypergraph matching theorem) requires us to avoid an arbitrary set of $|U|\log^{O(1)}n$ edges, and again there are too many choices to just take a union bound.

The first problem here is the most difficult to overcome. The second problem can be overcome by splitting the edges of $G$ randomly into $t$ subgraphs $G_i$, $i\in [t]$, for some $t=\log^{O(1)}n$, and showing then that (with high probability) each of these has some (slightly weaker) expansion property (see \Cref{lem:partitionedgesintoexpanders}). When we look for edge disjoint collections of paths in each graph $G_i$ separately for the application of the Aharoni-Haxell hypergraph matching theorem,  by finding $\frac 1t$ fraction of the required paths in each graph $G_i$ we need to avoid fewer edges by expanding in each $G_i$ separately rather than $G$. Thus, we have fewer sets of edges over which to take a union bound, solving the second problem above. The splitting into subgraphs is done in Section~\ref{sec:edge-partition-of-expanders} after a brief discussion of the parameters (see also the discussion at the start of Section~\ref{sec:connectinexpander}), but effectively this approach works because the splitting is very efficient: we do not need to reduce the lower bound in \eqref{expandprop} by more than a factor of 4 despite splitting into polylogarithmically many subgraphs (the upper bound on $|F|$ for which this holds, though, will decrease quite a bit more).

Solving the first problem to get some weak expansion into $V_i$ is the crux of this paper, and is where our new perspective on robust sublinear expansion is critical. We will show that it follows from this new perspective that (as stated more precisely in \Cref{prop:well-expanding-core}), for such sets $U$ and $F$, we have either

\begin{enumerate}
    \item $N_{G-F}(U)$ is actually much larger than guaranteed by \eqref{expandprop}, or
    \item  there is a set $U'\subseteq U$ which is much smaller than $U$ but whose neighbourhood alone contains at least (essentially) $\frac{|U|}{3\log^2n}$ vertices in $N_{G-F}(U)$.
\end{enumerate}
Given this, a natural approach is to take a union bound over all `well-expanding' sets $U$, meaning that they fall under condition 1.\ above, to guarantee a constant fraction of their neighbourhood gets sampled into our random set $V_i$. Since the subset $U'$ from condition 2.\ is well-expanding this will guarantee us that it expands inside $V_i$.
We would now like to use the fact that $U'$ expands inside $V_i$ to conclude the same happens for the original set $U$ from condition 2.\ above containing $U'$. However, a major issue here is that in order to achieve this we would need a bound on $|(N_{G-F}(U')\setminus U)\cap V_i|$ to get the expansion for \emph{any} relevant set $U$ for which we use the well-expanding subset $U'$ -- having to bound a random variable depending on $U'$ and $U$ spoils our union bound approach over the smaller sets $U'$.
For the first expansion of $U$ into $V_i$ this is avoidable (by only considering such sets disjoint from $V_i$), but we need to expand multiple times to reach most of the vertices in $V_i$ and, after the first expansion, avoiding expanding sets that contain vertices in $V_i$ is unavoidable.

To get around this, when we identify the well-expanding set $U'$, we look at its successive neighbourhoods in $G-F$ and show that enough of these vertices are chosen to be in $V_i$ together with all the vertices along a path going back to $U'$ so that $U'$ expands via such short paths to reach more than half of the vertices of $V_i$. In an early draft of this work, this was shown by carefully analysing an intricate random process. Fortunately, however, we will instead give here a much easier proof by combining our new perspective on robust sublinear expansion with an adaptation of a clever application of the sprinkling method appearing in a very recent work of Tomon~\cite{tomon2022robust}. 
We defer a more detailed sketch for this part of the argument to \Cref{sec:connectinexpander}, in particular until after we have introduced in full our new perspective on robust sublinear expansion, which remains crucial for this new approach.

In total, though, this will allow us to answer question ii) positively. I.e., property~\ref{prop:sketch} holds if we are allowed to use all the edges of $G$ to make connections through $V_i$. 
Let us also stress an important point, which already played a role at various points in the above arguments and that is that the paths $P_{xy}$ we find will always be short, namely of length $\log^{O(1)} n$. This again plays an important role in answering the next question, namely finding an appropriate sparse subgraph $H_i\subseteq G$ with the same property, which we turn to next.

\textbf{iii) Finding sparse connecting skeletons.} Before we look for a subgraph $H_i\subseteq G$ with the property \ref{prop:sketch} and $n\log^{O(1)}n$ edges, can we even find any graph with these properties? A binomial random graph is a natural candidate for such a graph, and, indeed, if $H$ is a binomial random graph with vertex set $V(G)$ and edge probability $p=\omega\left(\frac{\log n}n\right)$ then it will have, with high probability, the property~\ref{prop:sketch} if we replace $H_i$ with $H$ and choose $V_i$ to be any fixed set of linear size. (We prove this as \Cref{lem:template}, with a larger than optimal value of $p$ for simplicity.) 
We then use $H$ as a \emph{template} to construct the sparse expanding skeleton $H_i$. We first sample our large random subset of vertices $V_i$, and use our answer to question ii) to guarantee that  property~\ref{prop:sketch} holds with high probability in a slightly stronger form where every vertex is allowed to appear in $O(\log^5n)$ pairs in $\mathcal{P}$. In particular, we will use it with $\mathcal{P}$ being the set of pairs of vertices making an edge of our template graph $H$, which we choose to be sparse and well-connected through $V_i$, as well as have maximum degree $O(\log^5n)$. For each edge $xy\in E(H)$, we find an $xy$-path $P_{xy}$ through $V_i$ with length $\log^{O(1)}n$ in $G$ so that all these paths are edge disjoint. We then let $H_i$ be the union of all these paths, noting that, as $H$ is sparse and the paths $P_{xy}$ are relatively short, $H_i$ is also relatively sparse. Then, given an arbitrary collection $\mathcal{P}$ of pairs to connect, we first find edge disjoint paths connecting them through $V_i$ in $H$, before replacing each edge $xy$ on one of these paths in $H$ with the corresponding path $P_{xy}$ through $V_i$. This creates a set of edge disjoint $xy$-\emph{walks} through $V_i$ in $H_i$ --- as each such walk contains an $xy$-path, we can find the paths required by property~\ref{prop:sketch}. Note that, when we do this for each $i\in [3]$, we need to ensure that the graphs $H_i$ we find are edge disjoint, but this is easy to do by reusing some of our previous work, splitting $G$ into a union of edge disjoint expanders $G_1,G_2,G_3$, before finding each subgraph $H_i$ in the respective subgraph $G_i$.

\subsection{Organisation of the paper}\label{sec:organisation}
In the rest of this section we will introduce some general preliminary results, including some concentration results in \Cref{sec:conc} and the Aharoni-Haxell hypergraph matching theorem in \Cref{subsec:AH-matching-thm}, before showing strongly expanding graphs (namely $\G(n,p)$) satisfy a certain strong connectivity property in \Cref{sec:template}. In \Cref{sec:expansion}, we introduce robust sublinear expansion and prove a number of useful properties of this type of expansion. 
In \Cref{sec:connectinexpander}, we establish that our weaker expansion implies a similar (though weaker) connectivity property as that used in \Cref{sec:template}. 
In \Cref{sec:logstarproof}, we use the machinery we developed to prove our main result, \Cref{thm:logstar}. Finally, in \Cref{sec:final}, we make some concluding remarks.

\subsection{Concentration inequalities}\label{sec:conc}
We will often use a basic version of Chernoff's inequality for the binomial random variable (see, for example, \cite{alon-spencer}).

\begin{theorem}[Chernoff's bound]\label{chernoff}
Let $n$ be an integer and $0\le \delta,p \le 1$. If $X \sim \bin(n,p),$ then, setting $\mu=\mathbb{E} X = np,$ we have
$$\P(X>(1+\delta) \mu) \le e^{-\delta^2\mu/3},\quad\quad\quad \text{ and }\quad\quad\quad \P(X<(1-\delta) \mu) \le e^{-\delta^2\mu/2}.$$
\end{theorem}

We will also make use of the following well-known martingale concentration result (see Chapter 7 of \cite{alon-spencer}).

\begin{lemma}\label{lem:mcd}
Suppose that $X:\prod_{i=1}^N\Omega_i\to \R$ is $k$-Lipschitz. Then, for each $t>0$,
\[
\mathbb{P}(|X-\mathbb{E} X|>t)\leq 2\exp\left(\frac{-2t^2}{k^2N}\right).
\]
\end{lemma}


\subsection{The Aharoni-Haxell hypergraph matching theorem and edge disjoint paths}\label{subsec:AH-matching-thm}
We will use the following hypergraph version of Hall's theorem due to Aharoni and Haxell, which is an immediate consequence of Corollary~1.2 in \cite{aharoni2000hall} (noting that we can add new, unique, vertices to each edge in the theorem to make the hypergraphs $\ell$-uniform). A \emph{matching} in a hypergraph is a collection of pairwise vertex disjoint edges.

\begin{theorem}\label{thm:hyperhall}
Let $r\in \N$, and let $H_1,\ldots,H_r$ be a collection of hypergraphs with at most $\ell$ vertices in each edge. Suppose that, for each $I\subseteq [r]$, there is a matching in $\bigcup_{i\in I}H_i$ containing more than $\ell(|I|-1)$ edges. Then, there is an injective function $f:[r]\to \bigcup_{i\in [r]}E(H_i)$ such that $f(i)\in E(H_i)$ for each $i\in [r]$ and $\{f(i):i\in [r]\}$ is a matching of $r$ edges.
\end{theorem}

We will use \Cref{thm:hyperhall} to find edge disjoint paths between vertex pairs, to show that a graph is well-connected under the following definition, recalling that a path \emph{through} $V$ is a path with all its internal vertices in $V$. 

\begin{defn}\label{defn:path-connected}
A graph $G$ is $(\ell,t)$-\emph{path connected} through a vertex subset $V \subseteq V(G)$ if, for any $\mathcal{P} \subseteq \binom{V(G)}{2}$ in which every vertex appears in at most $t$ pairs in $\mathcal{P}$, there are edge disjoint paths $P_{\{x,y\}}$, $\{x,y\}\in \mathcal{P}$, such that, for each $\{x,y\}\in \mathcal{P}$, $P_{\{x,y\}}$ is an $xy$-path through $V$ with length at most $\ell$.
\end{defn}

We denote by $\binom{V(G)}{2}$ the multiset of pairs of distinct vertices of $G$, so in particular the same pair may appear multiple times in the collection $\mathcal{P}$.
Typically, $t$ will be a small constant and $\ell$ will be at most polylogarithmic in the number of vertices.

Given a graph $G$, a vertex set $V\subseteq V(G)$ and a collection $\mathcal{P}\subseteq \binom{V(G)}{2}$, we translate the pair connectivity property into a hypergraph matching problem as follows. For some $\ell\in \N$, and each $\{x,y\}\in \mathcal{P}$, let $H_{\{x,y\}}$ be the hypergraph with vertex set $E(G)$ and add as an edge the set $E(P)$ for each $xy$-path $P$ in $G$ with interior vertices in $V$ and length at most $\ell$. If there is an injective function $f:\mathcal{P}\to \bigcup_{\{x,y\}\in \mathcal{P}}E(H_{\{x,y\}})$ such that $f(\{x,y\})\in E(H_{\{x,y\}})$ for each $\{x,y\}\in \mathcal{P}$ and $\{f(\{x,y\}):\{x,y\}\in \mathcal{P}\}$ is a matching, then, for each $\{x,y\}\in \mathcal{P}$, let $P_{\{x,y\}}$ be the path in $G$ with edge set $f(\{x,y\})$. By the definition of $H_{\{x,y\}}$, each path $P_{\{x,y\}}$ is an $xy$-path in $G$ with length at most $\ell$ and interior vertices in $V$, and, as $\{f(\{x,y\}):\{x,y\}\in \mathcal{P}\}$ is a matching, these paths are all edge disjoint. Therefore, in order to prove that $G$ is $(\ell,t)$-{path connected} through $V$, it suffices to take a general such collection $\mathcal{P}$, define the relevant hypergraphs $H_{\{x,y\}}$, and prove that the associated condition holds for an application of \Cref{thm:hyperhall}.


\subsection{Existence of well-connected sparse graphs}\label{sec:template}
Random graphs typically present a natural candidate for a sparse, well-connected, graph, and we use this to prove the existence of our template in \Cref{lem:template} (using the connectivity property in \Cref{defn:path-connected}). Similar properties have been studied before in random graphs for various applications (see e.g.\ \cite{broder1996efficient,haxell2001tree,glebov2013hamilton}), and our main lemma (Lemma~\ref{lem:template}) can be obtained as a (not quite immediate) corollary of Lemma 3.4.\ from \cite{montgomery2021spanning} combined with a multi-round exposure argument like the one we use below. 
We include a different proof of \Cref{lem:template} for completeness but also use this to introduce two intermediate results (\Cref{prop:connectonepair} and \Cref{lem:connectmanypairs}) that we later use in the same manner to prove our key technical result, \Cref{thm:pathconnect}. 

We first remind the reader that a path through a subset of vertices $V$ is a path whose internal vertices are all in $V$. Let us also introduce, given $U,V \subseteq V(G)$, the \emph{ball of radius $i$ around $U$ within $V$} which we will denote as $B^{i}_G(U,V)$, namely it is the set of vertices in $V$ which can be reached by a path through $V$ of length at most $i$ starting from a vertex in $U$. The starting vertex in $U$ is \emph{not} required to be in $V$ itself, as is usual with our definition of paths through a set. We do, however, only consider reachable vertices within $V$, so that $B^{i}_G(U,V) \subseteq V$. 

For \Cref{prop:connectonepair}, we take an expansion property of sets of size $t\in \N$ and use this to connect a pair of vertices from a set of $2t-1$ pairs (c.f.\ the collection $\mathcal{P}$ in \Cref{defn:path-connected}).

\begin{prop}\label{prop:connectonepair}
Let $1\leq \ell,t \leq n$. Let $G$ be an $n$-vertex graph and let $V\subseteq V(G)$ be of size $|V|\geq 4t-2$ such that, for every $U \subseteq V(G)$ with size $|U|=t$, we have $|B^\ell_{G}(U,V)|>\frac{|V|}2$. Let $x_1,\ldots,x_{2t-1},y_1,\ldots,y_{2t-1}$ be distinct vertices of $G$. 

Then, for some $j\in [2t-1]$, there is an $x_jy_j$-path in $G$ through $V$ with length at most $4\ell \log n$.
\end{prop}
\begin{proof} 
Let $I_x$ be the set of $i\in [2t-1]$ for which $|B^{2\ell\log n}_G(x_i,V)|\leq \frac{|V|}2$ and let $I_y$ be the set of $i\in [2t-1]$ for which $|B^{2\ell\log n}_G(y_i,V)|\leq \frac{|V|}2$. Note that the required path can be found if there exists some $j\in [2t-1]$ with $j\notin I_x$ and $j\notin I_y$, for then $B^{2\ell\log n}_{G}(x_j,V)$ and $B^{2\ell\log n}_{G}(y_j,V)$ each have size larger than $\frac{|V|}2$ and must therefore intersect. Thus, we can assume that there is no such $j$, and, therefore, without loss of generality, that $|I_x|\geq t$.

Let $r\geq 0$ be the largest integer 
for which there is a set $X\subseteq \{x_i:i\in I_x\}$ for which $|X|\leq t\left(\frac23\right)^r$ and $|B^{(r+1)\ell}_G(X,V)|>\frac{|V|}2$, and let $X$ be any such set. Note that this is possible as $|I_x|\geq t$ and any subset of $\{x_i:i\in I_x\}$ with size $t$ satisfies these conditions for $r=0$ by the assumption of the proposition. 

Now, as $X\neq \emptyset$, we have that $t\left(\frac23\right)^r\geq 1$, so that $r\leq 2\log t\leq 2\log \frac{n}2$ (using $\left(\frac23\right)^2<\frac12$ and $n\geq |V|\geq 4t-2\geq 2t$) and thus $r< (2\log n)-1$. Thus, by the definition of $I_x$, $|B^{(r+1)\ell}_G(x_i,V)|\leq |B^{2\ell\log n}_G(x_i,V)|\leq \frac{|V|}2$ for each $i\in I_x$, and hence $|X|\geq 2$.
This allows us to partition $X=X_0\cup X_1$ with $|X_0|,|X_1|\leq \frac23 |X|\leq t\left(\frac23\right)^{r+1}$. As
\[
|B^{(r+1)\ell}_G(X_0,V)|+|B^{(r+1)\ell}_G(X_1,V)|\geq |B^{(r+1)\ell}_G(X_0\cup X_1,V)|=|B^{(r+1)\ell}_G(X,V)|>\frac{|V|}2\geq 2t-1,
\]
we can pick $j\in [2]$ such that $|B^{(r+1)\ell}_G(X_j,V)|\geq t$. Therefore, using the expansion of $t$-sets into $V$, we have $|B^{(r+2)\ell}_G(X_j,V)|>\frac{|V|}2$, contradicting the maximality of $r$ as $|X_j|\leq t\left(\frac23\right)^{r+1}$.
\end{proof}

 We now take a stronger expansion property and use \Cref{prop:connectonepair} in combination with the Aharoni-Haxell hypergraph matching theorem (\Cref{thm:hyperhall}) to find edge disjoint paths connecting a set of vertex pairs  (see also the discussion after \Cref{thm:hyperhall}), proving \Cref{lem:connectmanypairs}.

\begin{lemma}\label{lem:connectmanypairs}
Let $n\geq 2$ and $1\leq \ell,k \leq n$. Let $G$ be an $n$-vertex graph and let $V\subseteq V(G)$ so that $|V|\ge \frac n8+1$ and suppose, for each $U\subseteq V(G)$ and $F\subseteq E(G)$ with $U\neq\emptyset$ and $|F|\leq 2^9k|U|(\ell\log n)^2$, we have $|B^\ell_{G-F}(U,V)|>\frac{|V|}2$.

Then, $G$ is $(4\ell\log n,k)$-path connected through $V$.
\end{lemma}
\begin{proof} To show that $G$ is $(4\ell\log n,k)$-path connected through $V$, let $\mathcal{P}\subseteq \binom{V(G)}{2}$ be an arbitrary collection of vertex pairs such that each vertex appears in at most $k$ pairs in $\mathcal{P}$. Let $r=|\mathcal{P}|$, and order the pairs in $\mathcal{P}$. For each $i\in [r]$, let $H_i$ be the hypergraph with vertex set $E(G)$ and edge set corresponding to the edge sets of paths through $V$ with length at most $4\ell\log n$ connecting the $i$-th pair of vertices in $\mathcal{P}$, noting that the size of any edge is at most $4\ell\log n.$  

Now, for each $I\subseteq [r]$, let $M_I$ be a maximal matching in $\bigcup_{i\in I}H_i$. We will show that $|M_I|\geq 4\ell \log n \cdot |I|$ for each $I\subseteq [r]$. Towards a contradiction, suppose that, for some $I\subseteq [r]$, we have $|M_I|<4\ell \log n \cdot |I|$, noting that we must have $I\neq \emptyset$. Let $F$ be the set of edges of $G$ in any path corresponding to an edge in $M_I$, so that $F=\bigcup_{e\in M_I}V(e)$. Note that $|F|< 4\ell\log n \cdot 4\ell \log n \cdot |I| = (4\ell \log n)^2 |I|$. Let $I'$ be a maximal subset of $I$ such that no vertex appears in a pair in $I'$ more than once, so that $2k|I'|\geq |I|$. Note also that this ensures that $|I'|\le \frac n2$. Now, let $t=\ceil{|I'|/16}\geq 1$, so that $2t-1 \le 16t-15\le |I'|\le 16t$ and $4t-2 \le  \ceil{\frac n8}+1\le|V|$. Note that $|F|< (4\ell\log n)^2|I|\leq k\frac{|I'|}2(8\ell\log n)^2\leq 8kt(8\ell\log n)^2$, and therefore any set $U\subseteq V(G)$ with size $t$ satisfies $|B^\ell_{G-F}(U,V)|>\frac{|V|}2$.
Then, by \Cref{prop:connectonepair} applied to $G-F$, for some $j\in I'$ there is a path in $G-F$ between the $j$-th pair in $\mathcal{P}$ with interior vertices in $V$ and length at most $4\ell \log n$. Such a path corresponds to an edge of $H_j$ with no vertices in $F$, a contradiction to the maximality of $M_I$. Thus, we must have $|M_I|\ge 4\ell\log n  \cdot |I|$.

Therefore, by \Cref{thm:hyperhall}, there is a set of paths $P_i$, $i\in [r]$, in $G$ with $E(P_i)\in H_i$ for each $i\in [r]$, such that $E(P_i)$, $i\in [r]$, form edge disjoint sets. Thus, $G$ is  $(4\ell\log n,k)$-path connected through $V$.
\end{proof}

We now prove the existence of our template graph, by showing an appropriate expansion condition is likely in a certain binomial random graph and applying \Cref{lem:connectmanypairs}.

\begin{lemma}\label{lem:template}
For any large enough $n$, there exists an $n$-vertex graph $G$ with $\Delta(G)\leq 2^{8}\log^5 n$ and a set $V\subseteq V(G)$ with $|V|=\frac n6$ such that $G$ is $\left(\frac14\log^2 n,2\right)$-path connected through $V$.
\end{lemma}
\begin{proof}
Let $p=\frac{150\log^5 n}{n}$, let $V\subseteq [n]$ be a set of size $\frac n6$ and let $G \sim \G\left(n,p\right)$. As $p=\omega\left(\frac{\log n}{n}\right)$, a standard application of Chernoff's inequality (\Cref{chernoff}) shows that, with high probability, $\Delta(G)\leq \frac32pn= 225\log^5 n$. Therefore, it is sufficient to show that, with high probability, $G$ is $\left(\frac14\log^2 n,2\right)$-path connected through $V$.

\begin{claim*}
With high probability, for each $U\subseteq V(G)$ and each $F\subseteq E(G)$ with  $U\neq\emptyset$ and $|F|\leq 4|U|\log^4 n$ we have 
\begin{equation}\label{eqn:U}
|B_{G-F}(U,V)| > \min\left\{2^{16}|U|,\frac {|V|}{2}\right\}.
\end{equation}
\end{claim*}
\begin{proof} For large $n$, we will show for each $U\subseteq V(G)$ and $F\subseteq E(G)$ with  $U\neq\emptyset$ and $|F|\leq 4|U|\log^4n$ that \eqref{eqn:U} holds with probability at least $1-2^{-10|U|\log^5n}$, so that \eqref{eqn:U} holds for all $U\subseteq V(G)$ and $F\subseteq E(G)$ with $|F|\leq 4|U|\log^4n$ with probability at least
\[
1-\sum_{u=1}^{n} \binom{n}{u}\sum_{f=0}^{4u\log^4n}\binom{\binom{n}{2}}{f} \cdot  2^{-10u\log^5 n}\geq 1-\sum_{u=1}^nn^u\cdot n^2\cdot n^{8u\log^4n}\cdot 2^{-10u\log^5 n}\geq 1-\sum_{u=1}^nn^{-2}= 1-n^{-1},
\]
and the claim holds.

Let then $U\subseteq V(G)$ with $u=|U|\geq 1$ and $F\subseteq E(G)$ with $f=|F|\leq 4u\log^4 n$. Note that if $|B_{G-F}(U,V)| \leq  \min\left\{2^{16}u,\frac {|V|}{2}\right\}$, then there is some set $X=B_{G-F}(U,V)$ with size at most $2^{16}u$ such that $|V\setminus  X|\geq \frac{|V|}2$ and there are no edges between $U$ and $V\setminus  X$ other than in $F$. The probability of such a set $X$ existing is at most
\[
\sum_{i=0}^{2^{16}u}\binom{n}{i}\cdot (1-p)^{u\cdot \frac{|V|}{2}-|F|}\leq n^{2^{16}u+1}\cdot e^{-p(un/12-f)}\leq  e^{-pun/15}\leq 2^{-10u\log^5n},
\]
as required.
\end{proof}

With high probability then, we have that the conclusion of the claim holds for $G$. For any $U\subseteq V(G)$ and $F\subseteq E(G)$ with  $U\neq\emptyset$ and  $|F| \le 4|U| \log^4 n$, we have then by induction that, for each $i\geq 0$,
\[
|B_{G-F}^i(U,V)| > \min\left\{2^{16i}|U|,\frac{|V|}2\right\}.
\]
Setting $i=\ell=\frac{\log n}{16}$, we thus have $|B^{\ell}_{G-F}(U, V)|> \frac{|V|}2.$ Thus, by \Cref{lem:connectmanypairs} with $\ell=\frac{\log n}{16}$ and $k=2$, we have that $G$ is $\left(\frac14\log^2 n,2\right)$-connected through $V$. 
\end{proof}

\section{Robust sublinear expansion}\label{sec:expansion}
In this section, we will explain the expansion we use and its background, our new perspective on this type of expansion, and prove some key results using the expansion. Before doing so, for convenience, we state the definition of expansion that we use, as follows.

\begin{defn} \label{defn:robust-sublinear-expansion}
An $n$-vertex graph $G$ is an $(\eps,s)$-expander if, for every $U\subseteq V(G)$ and $F\subseteq E(G)$ with $1\le |U|\leq \frac23n$ and $|F|\leq s|U|$, we have
\begin{equation}
|N_{G-F}(U)|\geq \frac{\eps|U|}{\log^2 n}.\label{eqn:expands}
\end{equation}
\end{defn}

As the bound on the size of the neighbourhood guaranteed at \eqref{eqn:expands} is $o(|U|)$ as $n\to \infty$, we consider this to be \emph{sublinear} expansion. We often use Definition~\ref{defn:robust-sublinear-expansion} when $s$ is polylogarithmic in $n$, so that the set of edges $F$ may be of size $\omega(|U|)$ as $n\to\infty$, and call this \emph{robust} sublinear expansion. Note that, for each $v\in V(G)$, setting $U=\{v\}$ and setting $F$ to be the set of edges incident to $v$, we have that \eqref{eqn:expands} does not hold, so we must have $|F|>s$. Thus, $\delta(G) > s$ for any $(\eps,s)$-expander $G$, and therefore the expanders we will work with will always have minimum degree at least polylogarithmic in $n$. We discuss this notion of expansion in more detail with the relevant background in Section~\ref{sec:generalexpansionchat}. In Section~\ref{subsec:alternative}, we then introduce an alternative perspective of this expansion. In \Cref{sec:decomp-into-expanders}, we prove a lemma which almost decomposes an arbitrary graph into expanders. Finally, in \Cref{sec:edge-partition-of-expanders}, we prove sublinear expanders can be (edge) partitioned into expanders (with only slightly weaker expansion parameters).

\subsection{Expansion}\label{sec:generalexpansionchat}
Classical graph expansion is an immensely powerful idea in graph theory and computer science that has seen a very wide variety of applications (see, for example, the survey \cite{expander-survey}). A typical such property in a graph $G$ says that $|N_G(U)|\geq \lambda|U|$ for any set $U\subseteq V(G)$ which is not too large, where $\lambda$ is some strictly positive absolute constant not dependent on the graphs $G$ considered, though other notions have been considered instead of requiring a large $|N_G(U)|$, for example bounding the number of edges in $G$ between $U$ and $V(G)\setminus U$ (as indeed used by Conlon, Fox and Sudakov~\cite{conlon2014cycle}). 
Sublinear expansion is a weaker notion of this classical expansion introduced by Koml\'os and Szemer\'edi~\cite{K-Sz-1,K-Sz-2}, where we take a much smaller value of $\lambda$, but which is significant as every graph contains a sublinear expander $H$ with $\lambda=\Theta(1/\log^2|H|)$ (and even has a nice decomposition into sublinear expanders, as we will prove and use). Koml\'os and Szemer\'edi used sublinear expansion to find minors in sparse graphs, and more recently, sublinear expansion has found a host of other applications (see, for example, \cite{shapira2015small,montgomery2015logarithmically,fernandez2022nested,fernandez2022build,chakraborti2021well,haslegrave2021crux,haslegrave2021ramsey,kim2017komlos,liu2017mader,liu2020solution,liu2020clique,letzter2023immersion}).

Such sublinear expansion in a graph $G$ has some very weak \emph{robustness} properties, in that if $\lambda|U|/2$ vertices in $V(G)\setminus U$ are removed from the graph then the set $U$ will still expand (with the neighbourhood of $U$ still having at least $\lambda|U|/2$ vertices), and this property is used in many of the applications of sublinear expansion cited above. 
However, we will distinguish \emph{robust sublinear expansion} to be that where $U$ expands despite the removal of any set $F$ of at most $s|U|$ edges in $G$, where $s$ grows with $|G|$ so that this bound is superlinear in $|U|$, as in \Cref{defn:robust-sublinear-expansion}. Such robust sublinear expansion has recently been developed essentially independently by groups of different authors, appearing in some form in work by Haslegrave, Kim, and Liu~\cite{haslegrave2021extremal} and by Sudakov and Tomon~\cite{sudakov2022extremal} (with the parallel clearer in the expansion used in subsequent work by Jiang, Methuku and Yepremyan~\cite{jiang2021rainbow} and by Tomon~\cite{tomon2022robust}).
Roughly speaking, the expansion we use, as given in Definition~\ref{defn:robust-sublinear-expansion}, is a slightly weaker version of that used by Haslegrave, Kim, and Liu~\cite{haslegrave2021extremal} (so that we can find an almost-decomposition into such expanders) and a stronger version than subsequent developments of the expansion used by Sudakov and Tomon~\cite{sudakov2022extremal} (which makes it more powerful when we use it).

\subsection{An alternative notion of robustness}\label{subsec:alternative}
In \Cref{defn:robust-sublinear-expansion}, we consider the expansion to be robust as sets expand despite an arbitrary removal of a small number of edges. This can be alternatively encoded by recording that every vertex subset $U$ either expands very well (by some factor greater than $1$) or that its `robust neighbourhood' of vertices with plenty of edges towards $U$ expands well, though perhaps sublinearly (as in Proposition~\ref{prop:expansion-red-blue} below). It will be convenient to define this robust neighbourhood for any parameter $d$ as $$N_{G,d}(U):=\{v \in V(G) \setminus U : |N_G(v) \cap U| \ge d\},$$
that is, the set of vertices in a graph $G$, outside of a subset of vertices $U$, which have degree at least $d$ towards $U$.

\begin{prop}\label{prop:expansion-red-blue}
Let $G$ be an $n$-vertex $(\eps,s)$-expander, $U\subseteq V(G), |U|\le \frac23 n$ and $F$ a set of at most $s|U|/2$ edges. Then, for any $0<d \le s$, either
\[
\text{\textbf{\emph{a)}} }\;\;|N_{G-F}(U)| \ge \frac{s|U|}{2d},\;\;\;\text{ or }\;\;\;\text{\textbf{\emph{b)}}}\;\; |N_{G-F,d}(U)|\ge \frac{\eps|U|}{\log^{2} n}.
\]
\end{prop}
\begin{proof} Suppose \textbf{a)} is not satisfied, so that $|N_{G-F}(U)|<\frac{s|U|}{2d}$. Let $X=N_{G-F}(U)\setminus N_{G-F,d}(U)$, so that $|X|<\frac{s|U|}{2d}$. Let $F'$ be the edges of $G-F$ between $U$ and $X$, so that $|F'| < |X|d\leq s|U|/2$, and hence $|F|+|F'|\leq s|U|$. Note that, by the definition of $F'$, we have $N_{G-F,d}(U)=N_{G-F-F'}(U)$. As $G$ is an $(\eps,s)$-expander, we thus have
\[
|N_{G-F,d}(U)|=|N_{G-F-F'}(U)|\geq \frac{\eps|U|}{\log^{2} n},
\]
and therefore \textbf{b)} holds, as required.
\end{proof}

The following proposition shows more structure can be found in both outcomes of the above proposition. 
Though a more general variant follows easily, the parameters are tailored for our intended application.

\begin{prop}\label{prop:red-blue-expansion-robust-both} There is an $n_0$ such that the following holds for each $n\geq n_0$, $1 \ge \eps\ge2^{-9}$ and $s \ge 8\log^{13} n$. Let $G$ be an $n$-vertex $(\eps,s)$-expander, let $U\subseteq V(G)$ have size $|U|\le \frac23 n$ and let $F$ be a set of at most ${s|U|}/4$ edges. Then, in $G-F$ we can find either
\begin{enumerate}[label = \emph{\textbf{\alph{enumi})}}]
    \item $\frac{|U|}{\log^{7} n}$ vertex disjoint stars, each with $\log^{9} n$ leaves, centre in $U$ and all its leaves in $V(G)\setminus U$, or
    \item a bipartite subgraph $H$ with vertex classes $U$ and $X\subseteq V(G) \setminus U$ such that
    \begin{itemize}
        \item $|X| \ge \frac{\eps|U|}{2\log^2 n}$ and
        \item every vertex in $X$ has degree at least $\log^{4} n$ in $H$ and every vertex of $U$ has degree at most $2\log^{9} n$ in $H.$
    \end{itemize}  
\end{enumerate} 
\end{prop}

\begin{proof}
Take a maximal collection of vertex disjoint stars in $G-F$ with $\log^{9} n$ leaves and centre in $U$ and leaves outside of $U$. Let $C \subseteq U$ be the set of centres of these stars and $L\subseteq V(G)\setminus U$ be the set consisting of all their leaves. Assuming \textbf{a)} does not hold, we can thus assume that $|C| \le \frac{|U|}{\log^{7} n}$, $|L| \le |U| \log^2 n$, and, by the maximality, that there is no vertex in $U\setminus C$ with at least $\log^9 n$ neighbours in $G-F$ in $V(G)\setminus (U\cup L)$. Thus,
\begin{equation}\label{eqn:NGW}
|N_{G-F}(U\setminus C)|\leq |C|+|L|+|U\setminus C|\cdot \log^9n\leq  \frac{|U|}{\log^{7} n}+|U| \log^2 n+|U|\log^9n< 2|U|\log^9n.
\end{equation}

Let $d=\log^4n$ and $\Delta=2\log^{9} n$. We now construct the set $X\subseteq V(G)\setminus U$ and the bipartite subgraph $H$ through the following process, starting with $X_0=\emptyset$ and setting $H_0$ to be the graph with vertex set $U\cup X_0$ and no edges. Let $r=|V(G)\setminus U|$ and label the vertices of $V(G)\setminus U$ arbitrarily as $v_1,\ldots,v_r$. For each $i\geq 1$, if possible pick a star $S_i$ in $G-F$ with centre $v_i$ and $d$ leaves in $U$ such that $H_{i-1}\cup S_i$ has maximum degree at most $\Delta$, and let $H_i=H_{i-1}\cup S_i$ and $X_i=X_{i-1} \cup \{v_i\}$, while otherwise we set $H_i=H_{i-1}$ and $X_{i}=X_{i-1}$. Finally, let $H=H_r$ and $X=X_r=V(H_r)\setminus U$. We will show that \textbf{b)} holds for this choice of $H$ with bipartition $(U,X)$.

Firstly, observe that $\Delta(H_i)\leq \Delta$ for each $i\in [r]$ by construction, and that every vertex $v_i$ in $X$ has degree exactly $d$ in $H$, so the second condition in \textbf{b)} holds. Thus, we only need to show that $|X| \ge \frac{\eps|U|}{2\log^2 n}$ holds, which will follow as no vertex in $U\setminus C$ has $\frac{\Delta}2=\log^9n$ neighbours in $G-F$ in $X\setminus L$ due to the maximality of our family of disjoint stars defining $C$ and $L$.

Indeed, let $U'$ be the set of vertices in $U \setminus C$ with degree exactly $\Delta$ in $H$. As each vertex in $U' \subseteq U\setminus C$ has fewer than $\frac{\Delta}2$ neighbours in $G-F$ in $X\setminus L$, it must have at least $\frac{\Delta}2$ neighbours in $H$ in $X\cap L$. As each vertex in $X\cap L$ has $d$ neighbours in $H$, we have
\[
|U'|\leq \frac{d|X \cap L|}{\Delta/2}\le \frac{2d|L|}{\Delta}\leq \frac{2d\cdot |U|\log^2n}{\Delta}=\frac{2\log^4 n\cdot |U|\log^2n}{2\log^9n}\le \frac{\eps|U|}{8\log^2 n},
\]
where the last inequality follows for sufficiently large $n$.

Let $B=C \cup U'$, so that
\begin{equation}\label{eqn:Bbound}
|B|\leq \frac{|U|}{\log^{7} n}+\frac{\eps|U|}{8\log^2 n}\le \frac{\eps|U|}{6\log^2 n},
\end{equation}
and, in particular, $|U\setminus B|\geq \frac{2}3 |U|$.

Then, by \Cref{prop:expansion-red-blue} applied to $U\setminus B$ and $F$ with $d$, using that $|F| \le s|U|/4 \le s|U \setminus B|/2$, we have either $|N_{G-F}(U\setminus B)|\geq \frac{s|U\setminus B|}{2d}$ or $|N_{G-F,d}(U\setminus B)|\geq \frac{\eps |U\setminus B|}{\log^2 n}$. As 
$$\frac{s|U\setminus B|}{2d}\geq \frac{s|U|}{3d}\geq |U|+2|U|\log^9 n,$$ 
the former contradicts \eqref{eqn:NGW}, since $|N_{G-F}(U\setminus C)| \ge |N_{G-F}(U\setminus B)|-|U|$. Therefore, we must have that $|N_{G-F,d}(U\setminus B)|\geq \frac{\eps |U\setminus B|}{\log^2n}$. Every vertex $v_i$ in $N_{G-F,d}(U\setminus B) \setminus B$ has at least $d$ neighbours in $G-F$ in $U\setminus B$ which, not being in $B=U' \cup C$, by definition of $U'$ must all have degree strictly less than $\Delta$ in $H$. This implies $v_i \in X$, since we could add it together with some $d$ of these neighbours. Hence, we must have $N_{G-F,d}(U\setminus B)\setminus B\subseteq X$, and 
\[
|X|\geq |N_{G-F,d}(U\setminus B)|-|B|\overset{\eqref{eqn:Bbound}}{\geq} \frac{\eps |U\setminus B|}{\log^2 n} - \frac{\eps|U|}{6\log^2 n} \geq  \frac{\eps|U|}{2\log^2 n},
\]
as required, where we have used that $|U\setminus B|\geq \frac{2}3 |U|$.
\end{proof}

\subsection{Almost decomposing an arbitrary graph into expanders}\label{sec:decomp-into-expanders}

The following lemma almost decomposes a graph into robust sublinear expanders with, on average, very little overlap between their vertex sets. Setting $s=0$ in the below lemma (as we do in one application) obtains a full decomposition, although without any robustness.

\begin{lemma}\label{splitting-into-expanders}
Given an $n$-vertex graph $G$, a non-negative integer $s$ and $\eps \le 2^{-5}$ we can delete up to $4sn \log n$ edges from $G$ so that the remaining edges may be partitioned into graphs $G_1,\ldots, G_r$ such that $\sum_{i=1}^r |G_i| \le 2n$ and each $G_i$ is an $(\eps,s)$-expander.   
\end{lemma}

\begin{proof}
We prove this by induction on $n$, under the stronger condition that the graphs $G_1,\ldots,G_r$ in the partition satisfy $\sum_{i=1}^r |G_i| \le 2n-\frac{2n}{2+\log n}$. Since $2n-\frac{2n}{2+\log n}\ge n,$ and any $1$-vertex graph $G$ is trivially an $(\eps,s)$-expander, the lemma holds for $n=1$ with $G_1=G$. Let us then assume $n \ge 2$ and that the claim holds for all graphs with at most $n-1$ vertices. 

Letting $G$ be an $n$-vertex graph, note that, as $2n-\frac{2n}{2+\log n}\ge n,$ if $G$ is an $(\eps,s)$-expander then the trivial partition of $G_1=G$ demonstrates the claim holds for $G$. Thus, we can assume $G$ is not an $(\eps,s)$-expander, and in particular, that there exists a non-empty set of vertices $U \subseteq V(G)$ with $|U| \le \frac23n$ and a set $F$ of at most $s|U|$ edges such that $|N_{G-F}(U)| < \frac{\eps |U|}{\log^2 n}.$ Let $G_1=G[U\cup N_{G-F}(U)]-F$ and let $G_2=G\setminus U-E(G_1)-F$, so that $G_1$ and $G_2$ form an edge partition of $G-F$ and, setting $n_1=|G_1|$ and $n_2=|G_2|$, we have
\begin{equation}\label{eqn:n1n2}
n_1+n_2=|G_1|+|G_2|=|G|+|N_{G-F}(U)| < n+\frac{\eps |U|}{\log^2 n}\le n+\frac{\eps n_1}{\log^2 n}.
\end{equation}

Now, $n_2 = n-|U|<n$ and 
\begin{equation}\label{eqn:n1}
n_1\leq |U|+\frac{\eps |U|}{\log^2 n}\leq \frac23n+\eps n\leq \frac34n<n,
\end{equation}
so there exist sets $E_1\subseteq E(G_1)$ and $E_2\subseteq E(G_2)$ and partitions $G_{1,1},\ldots,G_{1,r_1}$ and $G_{2,1},\ldots,G_{2,r_2}$ of $G_1-E_1$ and $G_2-E_2$ into edge disjoint $(\eps,s)$-expanders so that, for each $i\in [2]$, $|E_i|\leq 4sn_i\log n_i$, and
\[
\sum_{j=1}^{r_i}|G_{i,j}|\leq 2n_i-\frac{2n_i}{2+\log n_i}.
\]
Therefore, we can remove $F\cup E_1\cup E_2$ from $G$ and decompose the remaining edges into $(\eps,s)$-expanders  $G_{1,1},\ldots,G_{1,r_1}$, $G_{2,1},\ldots,G_{2,r_2}$. We need then only check that $|F\cup E_1\cup E_2|\leq 4sn\log n$ and that the sum of the vertices of the expanders in this decomposition is at most $2n-\frac{2n}{2+\log n}$.

Firstly, note that, from \eqref{eqn:n1}, we have $\log n_1\leq \log \frac34n<\log n-\frac 25$, so that
\begin{align}
   \frac{1}{s}(|F|+|E_1|+|E_2|)&\leq  |U|+4n_1\log n_1+4n_2\log n_2
   \leq n_1+4n_1\left(\log n-\frac{2}{5}\right)+4n_2\log n \nonumber\\
   & = 4(n_1+n_2)\log n -\frac{3}{5}n_1 \overset{\eqref{eqn:n1n2}}{\leq} 4\left(n+\frac{\eps n_1}{\log^2 n}\right)\log n -\frac{3}{5}n_1 \nonumber\\
   & \leq 4n\log n.\label{eqn:EEE}
\end{align}

Secondly, again as $ \log \frac34n<\log n-\frac 25$,  we have
\begin{equation}\label{eqn:n1logs}
\frac{2n_1}{2+\log n_1} \overset{\eqref{eqn:n1}}{\ge} \frac{2n_1}{2+\log (3n/4)} \ge \frac{2n_1}{8/5+\log n}=\frac{2n_1}{2+\log n}+\frac{2n_1\cdot 2/5}{(8/5+\log n)(2+\log n)} > \frac{2n_1}{2+\log n}+\frac{n_1}{10\log^2 n},
\end{equation}
so that
\begin{align*}
\sum_{i=1}^2\sum_{j=1}^{r_i}|G_{i,j}|&\leq 2n_1+2n_2-\frac{2n_1}{2+\log n_1}-\frac{2n_2}{2+\log n_2}\overset{\eqref{eqn:n1logs}}{<} (n_1+n_2)\left(2-\frac{2}{2+\log n}\right)-\frac{n_1}{10\log^2 n} \\
&\overset{\eqref{eqn:n1n2}}{\leq} \left(n+\frac{\eps n_1}{\log^2 n}\right)\left(2-\frac{2}{2+\log n}\right)-\frac{n_1}{10\log^2 n}\leq 2n-\frac{2n}{2+\log n}.
\end{align*}
In combination with \eqref{eqn:EEE}, this shows that $G$ has the required decomposition, completing the inductive step and hence the proof.
\end{proof}


\subsection{Decomposing an expander into many expanders}\label{sec:edge-partition-of-expanders}

The following lemma partitions the edges of an expander into a chosen number of expanders with the same vertex set (and a slightly weaker expansion condition). Our expanders have two parameters, and it is key that this splitting is particularly efficient with the first parameter (as further discussed at the start of Section~\ref{sec:connectinexpander}) -- that is, we will split an $(\eps,s)$-expander into polylogarithmically many $(\eps/4,s')$-expanders for some appropriate $s'$, so that, while $s$ will be reduced polylogarithmically to get $s'$, the first parameter, $\eps$, is only reduced to $\eps/4$.

\begin{lemma}\label{lem:partitionedgesintoexpanders} Let $n,k,s\in \N$ and $0<\eps\leq 1$.
Suppose that $G$ is an $n$-vertex $(\eps,s)$-expander and $s\geq 2^{12}\eps^{-1} k^2\log^4n$. Then, there are edge disjoint graphs $G_1,\ldots,G_k$ such that $E(G)=\bigcup_{i\in [k]}E(G_i)$ and, for each $i\in [k]$, $G_i$ is an $\left(\frac{\eps}4,\frac{\sqrt{s\eps}}{ 8k \log n}\right)$-expander  with vertex set $V(G)$.
\end{lemma}
\begin{proof} 
If $n=1$, then the claim is trivially true, so let us assume that $n \ge 2$. Furthermore, observe that as $\eps>0$, we must have $s\leq \delta(G)$ for otherwise, we can remove all the neighbours of a vertex with minimum degree by removing at most $s$ edges, contradicting that $G$ is an $(\eps,s)$-expander, so certainly $s\leq n$.

Let $H$ be a random subgraph of $G$ with vertex set $V(G)$, which contains every edge independently with probability $\frac 1k$. Then, assign every edge of $G$ to one of the graphs $G_1,\ldots, G_k$ uniformly and independently at random, so that each $G_i$ is a random subgraph with the same distribution 
as $H$. Letting $s'=\frac{\sqrt{s\eps}}{ 8k \log n}\geq 8\log n$, we will show that the probability $H$ is not an 
$\left(\frac{\eps}4,s'\right)$-expander is strictly less than $\frac 1k$. Thus, by a union bound, the probability that each 
$G_i$ is an $\left(\frac{\eps}4,s'\right)$-expander is strictly positive, so some decomposition as required by the lemma must exist.

To show that $H$ is not an $\left(\frac{\eps}4,s'\right)$-expander with probability less than $\frac 1k$, we will take a union bound over all subsets $U$ of $V(G)$ for the event that $U$ fails the conditions of $(\frac{\eps}{4},s')$-expansion in $H$. For this, set $d=\sqrt{s/\eps} \log n$ and note that $\frac{s}{d}=8ks'$, $s'=\frac{\eps d}{8k \log^2 n}$ and $s\geq \sqrt{s}\cdot \sqrt{\eps^{-1}\log^2n}=d\geq 64 k$.

Let $U\subseteq V(G)$ and $u=|U|\leq \frac{2n}{3}$. By \Cref{prop:expansion-red-blue} with $d$, $U$ and $F=\emptyset$, we have either \textbf{a)} $|N_G(U)|\geq \frac{su}{2d}=4ks'u$ or \textbf{b)} $|N_{G,d}(U)|\geq \frac{\eps u}{\log^2 n}$. 

If \textbf{a)} holds, then $|N_H(U)|$ is dominated by $\text{Bin}(4ks'u,1/k)$, so that $\P(|N_H(U)|\geq 2s'u)\geq 1-e^{-s'u/2}$ by a Chernoff bound (\Cref{chernoff}). Note that if $|N_H(U)|\geq 2s'u$ then for any $F\subseteq E(H)$ with $|F|\leq s'u$ we have $|N_{H-F}(U)|\geq s'u\geq \frac{\eps u}{4\log^2n}$. Thus, when \textbf{a)} holds, the probability $U$ fails the $\left(\frac{\eps}4,s'\right)$-expansion condition is at most $e^{-s'u/2}\leq e^{-4u\log n}$ as $s'\geq 8\log n$.

If \textbf{b)} holds, then we have $N_{G,d}(U)\neq \emptyset$, so in particular $u=|U| \ge d$, which in turn implies $\frac{\eps u}{\log^2n}\ge \frac{\sqrt{s\eps}}{\log n}\ge 64k\log n$. Note now that the probability that any vertex $v\in N_{G,d}(U)$ is not in $N_{H,d/2k}(U)$ is at most $p:=\P\left(\text{Bin}\left(d,\frac1k\right)< \frac{d}{2k}\right)$, where we have once again by a Chernoff bound that
\begin{equation}\label{eqn:p}
    p\leq e^{-d/8k}\le \frac 14.
\end{equation}
Therefore, if we write $t=\frac{\eps u}{2\log^2 n}$, as $|N_{G,d}(U)|\geq \frac{\eps u}{\log^2 n}=2t$, we have
\begin{align}\label{eqn:NHU}
\P\left(|N_{H,d/2k}(U)|< \frac{\eps u}{2\log^2n}\right)&\le \binom{\ceil{2t}}{\ceil{2t}-\floor{t}}\cdot p^{\ceil{2t}-\floor{t}}  \le 2^{\ceil{2t}}\cdot p^{\ceil{2t}-\floor{t}}\le (4p)^t \overset{\eqref{eqn:p}}{\le} e^{\frac{-\eps u d}{32k\log^2 n}}= e^{-s'u/4}\notag\\&\le e^{-2u\log n}.
\end{align}
Note that, if $|N_{H,d/2k}(U)|\geq \frac{\eps u}{2\log^2n}$, then, for any $F\subseteq E(H)$ with $|F|\leq s'u$, we have 
\[
|N_{H-F}(U)|\geq |N_{H,d/2k}(U)|-\frac{|F|}{d/2k}\geq \frac{\eps u}{2\log^2n}-\frac{2ks'u}{d}=\frac{\eps u}{2\log^2n}-\frac{su}{4d^2}= \frac{\eps u}{4\log^2n}.
\]
Thus, \eqref{eqn:NHU} implies that the probability $U$ does not satisfy the $(\eps/4,s')$-expansion condition is at most $e^{-2u\log n}$.

Therefore, whichever of \textbf{a)} or \textbf{b)} holds, the probability $U$ does not satisfy the $(\eps/4,s')$-expansion condition is at most $e^{-2u\log n}$. Hence, the probability that $H$ is not an $(\eps/4,s')$-expander is at most
$$ \sum_{u=1}^{2n/3}\binom{n}{u}e^{-2u\log n} \le \sum_{u=1}^{2n/3} n^u\cdot n^{-2u} \le n^{-1}+\sum_{u=2}^{2n/3} n^{-2}\leq \frac{2}{n}< \frac{1}{k},$$
as required, where in the last inequality we make use of the observation that $n\geq s\geq 2^{12}\eps^{-1}k^2\log^4n>2k$.
\end{proof}


\section{Finding edge disjoint paths through random vertex sets}\label{sec:connectinexpander}

We will now show that a robust sublinear expander is not only well-connected (in the sense of Definition~\ref{defn:path-connected}, and as used in Section~\ref{sec:template}), but is likely to be well-connected through any large random vertex subset. That is, we prove the following result.

\begin{theorem}\label{thm:pathconnect} 
Let $G$ be an $n$-vertex $(\eps,s)$-expander with $1\ge \eps \ge 2^{-7}$ and $s 
\ge \log^{135} n$. Let $V\subseteq V(G)$ be a random subset chosen by including each vertex independently at random with probability $\frac13$. Then, with high probability, $G$ is $(4\log^5 n, 2^8\log^5 n)$-path connected through $V$. 
\end{theorem}

The challenge of proving Theorem~\ref{thm:pathconnect} is discussed in Section~\ref{subsec:sketch}, and in particular in the answer to question \textbf{ii)} there. Having since then proved Lemma~\ref{lem:connectmanypairs}, let us note that, for suitable polylogarithmic parameters $\bar{s}$ and $\ell$, by this lemma, to prove Theorem~\ref{thm:pathconnect} it is sufficient to show that, with high probability,

\vspace{-0.4cm}

\begin{enumerate}[label = $(\dagger)$]\item  $|B^\ell_{G-F}(U,V)|> \frac{|V|}2$ for each $U\subseteq V(G)$ and $F\subseteq E(G)$ with $|F|\leq \bar{s}|U|$.\label{youdag} 
\end{enumerate}

\vspace{-0.4cm}

Here, let us remind the reader that given $U,V \subseteq V(G)$, we denote by $B^{i}_G(U,V)$ the set of vertices in $V$ which can be reached by a path through $V$ (with all internal vertices in $V$) of length at most $i$ starting from a vertex in $U$.

To further motivate our approach in this section, let us recap some key aspects discussed in Section~\ref{subsec:sketch} with the parameters we have now introduced. As our condition is that $G$ is an $n$-vertex $(\eps,s)$-expander, we could have sets $U\subseteq V(G)$ with $|N_{G}(U)|=\eps|U|/\log^2n$, so that, with probability $(2/3)^{\eps|U|/\log^2n}$, we may have $N_G(U) \cap V=\emptyset$, so that \ref{youdag} does not even hold for $U$ with $F=\emptyset$. Thus, if we take a union bound approach to \ref{youdag}, we could reasonably only hope to bound this over $\exp(O(\eps u/\log^2n))$ pairs $(U,F)$ with $|U|=u$ (for any relevant $u$). 
As discussed in Section~\ref{subsec:sketch}, we solve this in two ways. Firstly, for our union bound, we consider for \ref{youdag} only sets $U$, which expand particularly well, much more than the lower bound from the expansion condition. By showing that a general set contains a well-expanding set that is not much smaller, we will then be able to deduce that all sets satisfy at least some weaker form of \ref{youdag}. Secondly, we randomly partition the edges of $G$ into expanders $G_1,\ldots,G_k$, for some appropriate polylogarithmic value of $k$, each of which will be an $(\eps/4,s')$-expander for some appropriate $s'$. As the first parameter in the expansion ($\eps$) has only reduced by a factor of 4, and this is the parameter that appears in the number of pairs for which we can hope to take a union bound, there is little cost to looking at expansion only within any given subgraph $G_i$. Then, when we want to expand, avoiding some larger set of edges $F$, with $|F|=\bar{s}|U|$ as in \ref{youdag}, we can find the smallest set of edges $F\cap E(G_i)$, and expand $U$ in $G_i$ while only avoiding a relatively much smaller set of edges $F\cap E(G_i)$. Thus, for our union bound, we only need to consider these smaller sets of edges, and this plays a crucial role in reducing the number of pairs $(U,F)$ that we need to consider when applying the union bound.

We do this all in three stages. Firstly, in Section~\ref{sec:expand1}, we show that for any set $U\subseteq V(G)$ with $|N_G(U)|\geq |U|\log^{24}n$ and any $F\subseteq E(G)$ with $|F|\leq |U|$ we have $|B^\ell_{G-F}(U,V)| > \frac{|V|}2$ holds with probability $1-\exp(-\Omega(|U|\log^{2}n))$. That is, we show the desired property holds for a \emph{well-expanding} set $U$ with high enough probability that we can take a union bound over all well-expanding sets $U$ (and any edge set $F$ with, say, $|F|\leq |U|$) to get that this property holds for all well-expanding sets with probability $1-o(1/n)$. As discussed in Section~\ref{subsec:sketch}, if the set $U$ does not expand well, then we can not guarantee it has the property we want with high enough probability for a union bound over all sets $U$, so we only consider well-expanding sets $U$ here. Then, in Section~\ref{sec:expand2}, we show that every set $U$ in our expander contains a well-expanding set $U'$ which is not that much smaller than $U$ (see \Cref{prop:well-expanding-core}), indeed, we will find such a $U'$ satisfying $|U'|\geq \mu|U|$, where $\mu\geq 1/\log^{27}n$. Thus, for any edge set $F$ with $|F|\leq \mu|U|\leq |U'|$ we have $|B^\ell_{G-F}(U,V)|\geq |B^\ell_{G-F}(U',V)| > \frac{|V|}2$. This is almost the condition we need to apply Lemma~\ref{lem:connectmanypairs}. However, $\mu$ is too small, namely by a polylogarithmic factor smaller than the value of $\bar{s}$ that we need. Therefore, in Section~\ref{subsec:well-connectivity-weak-expander}, we find the stronger expansion property we need by first splitting an $(\eps,s)$-expander into polylogarithmically many edge disjoint expanders via \Cref{lem:partitionedgesintoexpanders}, before showing that with high probability each one of these has the above-mentioned weaker expansion property into $V$. Finally, we combine these properties to show that, indeed, $G$ has the desired stronger expansion property into $V$. This allows us to apply Lemma~\ref{lem:connectmanypairs}, completing the proof of Theorem~\ref{thm:pathconnect}.


\subsection{Expansion of well-expanding sets into a random vertex set}\label{sec:expand1}

In an $n$-vertex $(\eps,s)$-expander $G$, given $U\subseteq V(G)$ such that $|N_{G}(U)|\geq |U|\log^{24}n$ and $F\subseteq E(G)$ with $|F|\leq s|U|/4$,  when $V\subseteq V(G)$ is chosen by selecting each vertex independently at random with probability $\frac 13$, we wish to show that, with some large probability, we have $|B^\ell_{G-F}(U,V)|>\frac{|V|}2$, for some appropriate parameters $\eps,s$ and $\ell$ (thus proving Lemma~\ref{lem:expandintorandom} below). To prove this we will adapt a `sprinkling' argument from a very recent work of Tomon~\cite{tomon2022robust}, thus avoiding a much more complex argument from initial versions of this work. To prove this we will reveal the vertices in $V$ in $\ell$ batches, using the so-called sprinkling method, by partitioning $V$ randomly into sets $V_1\cup\ldots\cup V_\ell$, weighted so that most of the vertices in $V$ are likely to be in $V_\ell$. A natural approach here would be to prove a likely bound on $B^i_{G-F}(U,V_1\cup\ldots \cup V_i)$ for each $i\in [\ell]$, resulting in a bound on $B^\ell_{G-F}(U,V_1\cup\ldots \cup V_\ell)=B^\ell_{G-F}(U,V)$, so let us emphasise that this is \emph{not} what we do.

Instead, for $0\leq i\leq \ell$ we track the size of sets $B_i$ which are defined as the set of vertices $v \in V(G)$ which can be reached from $U$ by a path all of whose internal vertices are in $V_1 \cup \ldots \cup V_{i-1}$ (i.e., a path through this set) of length at most $i$. It is \emph{crucial} here that we do not insist $v$ belongs to either $U$ or $V_1\cup \ldots \cup V_i$. One can think of vertices in $B_i$ as having potential of making all their neighbours in $G-F$ reachable in the same way (so being in $B_{i+1}$) if they get sampled into our next random subset $V_{i}$. In particular, every vertex with a neighbour in $B_i$ which gets sampled into $V_{i}$ will join $B_{i+1}$. This combined with our notion of robust expansion (as discussed further below) allows us to show that it is likely that $B_{i+1}$ will increase in size compared to $B_i$ until for some $i\le \ell-1$ its size is at least $\frac23 n$. In particular, we will have $B_{\ell} \ge \frac23 n$. The final stage is slightly different, here since $B_{\ell}$ is independent of our final random set $V_{\ell}$ we will likely have almost $\frac23$ of the vertices of $V_\ell$ belonging to $B_{\ell}$. As our random sets are weighted heavily towards $V_\ell$, it is likely to contain more than $\frac34$ of the vertices of $V$ so that we will likely have $|V_\ell\cap B_\ell|>\frac{|V|}2$, so that, finally, we have
\[
|B^\ell_{G-F}(U,V)|\geq |V_\ell\cap B_\ell|>\frac{|V|}2,
\]
as required.

That the sets $B_i$, $1\leq i\leq \ell-1$, are very likely to increase notably in size will follow from our notion of robust expansion (as proved in the claim below).
In particular, at step $i$, \Cref{prop:red-blue-expansion-robust-both} tells us that one of two cases \textbf{a)} or \textbf{b)} may occur.

\textbf{a)} $B_i$ has many large vertex disjoint stars extending from $B_i$. In this case we use that, for each centre sampled into $V_i$, the (many) corresponding leaves are added to $B_{i+1}$. We will have that many more leaves are added for each successful centre than the sampling probability for $V_i$, so that this is a good increase in size.

\textbf{b)} $B_i$ has a large robust neighbourhood whose vertices have many neighbours in $G-F$ in $B_i$. Each vertex in this robust neighbourhood is likely to have at least one of these neighbours in $B_i$ sampled into $V_i$, whereupon it will then be in $B_{i+1}$. (In fact, we need a slightly stronger property to hold so that the sampling of each vertex in $B_i$ does not have too strong an influence on the size of $B_{i+1}\setminus B_i$, which is why we use the subgraph $H$ provided by case \textbf{b)} of \Cref{prop:red-blue-expansion-robust-both}.)

Thus, in either case, $|B_i|$ is likely to increase.

\begin{lemma}\label{lem:expandintorandom} Let $n \ge 2.$ Suppose that $G$ is an $n$-vertex $(\eps,s)$-expander with $2^{-9}\leq \eps \leq 1$ and $s \ge 8\log^{13} n$. Let $U\subseteq V(G)$  satisfy $|N_{G}(U)|\geq |U|\log^{24}n$ and let $F\subseteq E(G)$ satisfy $|F|\leq {|U|}$. Let $V\subseteq V(G)$ be a random subset chosen by including each vertex independently at random with probability $\frac13$. 

Then, with probability $1-e^{-\Omega\left(|U|\log^{2}n\right)}$,
\begin{equation}\label{expand}
|B^{\log^4n}_{G-F}(U, V)|>\frac{|V|}{2}.
\end{equation}
\end{lemma}

\begin{proof} Let $\ell=\log^4n$, $q=\frac3{11}$ and let $p$ be such that $1-(1-p)^{\ell-1}(1-q)=\frac{1}{3}$, i.e., that $(1-p)^{\ell-1}=\frac{11}{12}$, so that
\begin{equation}\label{eqn:p15}
p\ge \frac{1}{15\log ^4 n}.
\end{equation}
Let $G$ be an $n$-vertex $(\eps,s)$-expander, $U\subseteq V(G)$ with $|N_G(U)| \ge |U|\log^{24} n$ and $F\subseteq E(G)$ with $|F|\leq {|U|}$. Independently, for each  $i\in \{1,\ldots,\ell\}$, let $V_i$ be a random subset of $V(G)$ with each vertex included independently at random with probability $p$ if $i\leq \ell-1$ and with probability $q$ if $i=\ell$. Set $V=V_1\cup\ldots \cup V_\ell$, and note that each vertex is included in $V$ independently at random with probability $\frac13$. Thus, we wish to show that, with  probability at least  $1-e^{-\Omega\left({|U|}{\log^{2}n}\right)}$ we have $|B^{\ell}_{G-F}(U , V)|>\frac{|V|}{2}$.

For each $0\le i\leq \ell$, let $B_{i}$ be the set of vertices of $G$ which can be reached via a path in $G-F$ which starts in $U$ and has length at most $i$ and whose internal vertices (if there are any) are in $V_1\cup \ldots \cup V_{i-1}$. In particular, then, we have $B_0=U$ and $B_1=B_{G-F}(U)$. 
Observe also that $B_0\subseteq B_1\subseteq \ldots \subseteq B_{\ell}$. We emphasise that the vertices of $B_{i}$ do not have to themselves be inside $V_1\cup \ldots \cup V_{i-1}$, only the interior vertices of some path from $U$ to the vertex in $B_{i}$ are required to be inside $V_1\cup \ldots \cup V_{i-1}$. An important property of $B_{i}$ is that it is completely determined by the sets $U,V_1,\ldots, V_{i-1}$, so is in particular independent of $V_{i}$. Note also that any vertex in $N_{G-F}(B_i)$ with a neighbour in $B_i$ that gets sampled into $V_i$ belongs to $B_{i+1}$. These two observations will be the key behind why the sets $B_{i+1}$ will grow in size until they occupy most of the set $V(G)$. In particular, finally, observe that
\begin{equation}\label{eqn:overkill4}
B_\ell\cap V_\ell\subseteq B^{\ell}_{G-F}(U, V).
\end{equation}

We now show that indeed, for each $1\le i\le \ell-1$, that, unless $B_{i}$ is already very large, $B_{i+1}$ is likely to be larger than $B_i$, as follows.

\begin{claim*} For each $1\le i\le \ell-1$, with probability $1-e^{-\Omega\left({|U|}{\log^{2}n}\right)}$, either $|B_i| \ge \frac23 n$, or  $$|B_{i+1}\setminus B_i| \ge \frac{\eps|B_i|}{2^6\log^{2} n}.$$
\end{claim*}
\begin{proof} For each $v\in N_{G-F}(B_i)$, $v$ is in $B_{i+1}$ if at least one of its neighbours in $G-F$ in $B_i$ gets sampled into $V_{i}$. That is,
\begin{equation}\label{eqn:overkill1}
\{v\in N_{G-F}(B_i):N_{G-F}(v,B_i)\cap V_i\neq\emptyset\} \subseteq B_{i+1}\setminus B_i. 
\end{equation}
We will show that, for any set $W\subseteq V(G)$ with $|W|\leq \frac23 n$ and $B_1\subseteq W$ 
\begin{equation}\label{eqn:overkill2}
\P\left(|\{v\in N_{G-F}(W):N_{G-F}(v,W)\cap V_i\neq\emptyset\}|\geq \frac{\eps|W|}{2^6\log^{2} n}\right)\ge 1-e^{-\Omega\left({|B_1|}/{\log^{22} n}\right)}.
\end{equation}
Thus, we will have for all $1\le i \le \ell -1$
\begin{align*}
\P\left(|B_i| \ge \frac23 n  \: \text{ or  } \: |B_{i+1}\setminus B_i| \ge \frac{\eps|B_i|}{2^6\log^{2} n}\right) &\overset{\textcolor{white}{\eqref{eqn:overkill1}}}{\geq} \P\left(|B_{i+1}\setminus B_i| \ge \frac{\eps|B_i|}{2^6\log^{2} n}\: \big| \: |B_i|\leq \frac{2}{3}n\right)
\\
&\overset{\eqref{eqn:overkill1}}{\geq} 
\P\left(|\{v\in N_{G-F}(B_i):N_{G-F}(v,B_i)\cap V_i\neq\emptyset\}|\ge \frac{\eps|B_i|}{2^6\log^{2} n}\: \big| \: |B_i|\leq \frac{2}{3}n\right)\\
&\overset{\eqref{eqn:overkill2}}{\ge} 
1-e^{-\Omega\left({|B_1|}/{\log^{22} n}\right)}\ge 1-e^{-\Omega\left({|U|}{\log^{2} n}\right)},
\end{align*}
where in the last inequality we used $|B_1|=|B_{G-F}(U)|\ge |U|\log^{24} n-|F|\geq \frac12|U|\log^{24}n$.

Let then $W\subseteq V(G)$ with $|W|\leq \frac23n$ and $B_1\subseteq W$. As $|W|\leq \frac23 n$, and $|F|\le |U|\leq |B_1|\leq |W|\le {s|W|}/4$ we can apply \Cref{prop:red-blue-expansion-robust-both} to $W$ and $F$ to show one of two cases \textbf{a)} or \textbf{b)} holds and we will show that \eqref{eqn:overkill2} holds in either case.

\textbf{a)} Suppose $G-F$ contains $\frac{|W|}{\log^{7} n}$ vertex disjoint stars with $\log^{9} n$ leaves, with the centre in $W$ and all leaves in $N_{G-F}(W)$. Let $C\subseteq W$ be the set of centres of such a collection of stars, and note that
\begin{equation}\label{eqn:overkill3}
|\{v\in N_{G-F}(W):N_{G-F}(v,W)\cap V_i\neq \emptyset\}|\geq |C\cap V_i|\log^9n.
\end{equation}
By a Chernoff bound (\Cref{chernoff}) and \eqref{eqn:p15}, with probability  at least 
$1-e^{-p|C|/8}= 1-e^{-\Omega\left({|W|}/{\log^{11}n}\right)}$, we have $|C\cap V_i|\geq \frac{p|C|}2\geq \frac{|W|}{2^6\log^{11}n}$. Thus, in combination with \eqref{eqn:overkill3}, we have that \eqref{eqn:overkill2} holds as $\eps\leq 1$.

\textbf{b)} Suppose instead that there is a bipartite subgraph $H\subseteq G-F$ with vertex classes $W$ and $X$ such that
    \begin{itemize}
        \item $|X| \ge \frac{\eps|W|}{2\log^2 n}$ and
        \item every vertex in $X$ has degree at least $\log^{4} n$ in $H$ and every vertex of $U$ has degree at most $\Delta:=2\log^{9} n$ in $H.$
    \end{itemize}  
    For each $v\in X$, the probability that $v$ has no neighbours in $H$ in $V_i$ is at most
    \[
    (1-p)^{\log^4n}\le e^{-p\log^4n}\overset{\eqref{eqn:p15}}{\leq} e^{-1/15}\leq \frac{15}{16}.
    \]
Let $Y$ be the random variable counting the number of vertices of $X$ having a neighbour in $V_{i}$ in $H$, so that $\E Y \ge \frac{|X|}{16}$. Observe also that $Y$ is $\Delta$-Lipschitz since for each $v\in W$ the event $\{v\in V_i\}$ affects $Y$ by at most $d_H(v)\leq \Delta$. 
Hence, by \Cref{lem:mcd} with $k=\Delta$, $t=\frac{|X|}{32}$ and $N=|W|$, we have
$$\P\left(Y < \frac{|X|}{32}\right)\le \P\left(Y < \E Y - \frac{|X|}{32} \right)\le2\exp\left(-\frac{2^{-9}|X|^2}{\Delta^2|W|}\right)= e^{-\Omega\left({|W|}/{\log^{22} n}\right)}. $$
Each vertex in $X$ with a neighbour in $V_i$ in $H$ lies in $\{v\in N_{G-F}(W):N_{G-F}(v,W)\cap V_i\neq\emptyset\}$, so therefore, with probability at least $1-e^{-\Omega\left({|W|}/{\log^{22} n}\right)}$, we have $|\{v\in N_{G-F}(W):N_{G-F}(v,W)\cap V_i\neq\emptyset\}|\geq Y\geq \frac{|X|}{32}\geq \frac{\eps|W|}{2^6\log^2 n}$ and thus \eqref{eqn:overkill2} holds as well in case \textbf{b)}, completing the proof.
\end{proof}

As $B_\ell$ and $V_\ell$ are independent and $q=\frac3{11}$ (so that $\frac{2q}{3}=\frac{2}{11}>\frac{4}{23}$), by Chernoff's bound (\Cref{chernoff}), we have that
\[
\P\left(|B_{\ell}\cap V_\ell|\le\frac{4}{23}n \: \big| \: |B_{\ell}|\geq \frac{2}3n\right)\leq \P\left(\bin\left(\frac{2}3n,q\right)\leq \frac{4}{23}n\right)\le e^{-\Theta(n)},
\]
and, similarly, as $\frac13<\frac8{23}$ we have $\P\left(|V|\geq \frac{8}{23}n\right)\le e^{-\Theta(n)}$.

Thus, by the claim, we have in total that

\begin{enumerate}[label=\roman*)]
\item \label{itm1}
for each $i\in [\ell-1]$, $|B_i| \ge \frac23 n$ or  $|B_{i+1}\setminus B_i| \ge \frac{\eps|B_i|}{2^6\log^{2} n}$, and

 \item $|B_\ell|<\frac23 n$ or $|B_\ell\cap V_\ell|>\frac{4}{23}n$, and

\item \label{itm3} $|V|\leq \frac{8}{23}n$
\end{enumerate}

with probability at least
$$1-\log^4n\cdot e^{-\Omega\left({|U|}{\log^{2} n}\right)}-2^{-\Theta(n)}\ge 1-e^{-\Omega\left({|U|}{\log^{2} n}\right)}.$$

However, if \ref{itm1}--\ref{itm3} all hold, then, for each $i\in [\ell-1]$, we have
$$|B_i| \ge \min\left\{\frac23 n,\left(1+\frac{\eps}{2^6\log^2 n}\right)^{i}|U|\right\}\geq \min\left\{\frac23 n,\exp\left(\frac{\eps i}{2^7\log^2 n}\right)\right\},$$
so that, setting $i=\ell=\log^4n$, we conclude $|B_{\ell}|\geq \frac23 n$, and hence, by ii) and iii), that $|B_\ell\cap V_\ell|>\frac{|V|}2$.

Thus, by \eqref{eqn:overkill4}, we have that $|B^{\ell}_{G-F}(U,V)|>\frac{|V|}2$ with probability at least $1-e^{-\Omega\left({|U|}{\log^{2} n}\right)}$.
\end{proof}


\subsection{Expansion into a random vertex set}\label{sec:expand2}

Having picked $V\subseteq V(G)$ with vertex probability $\frac13$ in an $n$-vertex $(\eps,s)$-expander $G$, \Cref{lem:expandintorandom} tells us that for any \emph{fixed}, well-expanding subset of vertices $U$ and small set $F$ of edges we can reach more than one half of the vertices of $V$ by short paths through $V$ in $G-F$ with pretty high probability. 
We now want to use this to show that a similar property holds \emph{simultaneously} for \emph{all} vertex subsets $U$. As we cannot directly take a union bound over all subsets $U$, we first show that any vertex subset $U$ in an $(\eps,s)$-expander contains a subset $U'\subseteq U$ which expands particularly well (relative to its own size), which is not much smaller than $U$, and so that $U'$ captures much of the guaranteed expansion of $U$ (as it will easily follow that $|N_G(U')|\geq \frac{\eps|U|}{3\log^2n}$). This follows easily from the definition of expansion but is perhaps easier to immediately see why it is true from the perspective introduced in \Cref{prop:expansion-red-blue}.

\begin{prop}\label{prop:well-expanding-core} Let $n\geq 2$, $0<\eps\leq 1$ and $s \ge \log^{24} n$. Let $G$ be an $n$-vertex $(\eps,s)$-expander and let $U\subseteq V(G)$ have size $|U|\le \frac{2}3n$. 

Then, there is a set $U'\subseteq U$ with $|N_{G}(U')|\geq |U'|\log^{24} n$ and $|U'|\geq \frac{\eps|U|}{3\log^{26}n}$.
\end{prop}
\begin{proof} If $U=\emptyset$, then $U'=\emptyset$ easily satisfies the required conditions, so we can assume that $|U|\geq 1$. 
Let, then, $U'\subseteq U$ be maximal subject to $|N_{G}(U')|\geq |U'|\log^{24}n$, noting this is possible as $U'=\emptyset$ satisfies these conditions. Suppose that $U\neq U'$, for otherwise $U$ satisfies the conditions itself. Then $|N_{G}(U')|< (|U'|+1)\log^{24}n+1$ or we could add an arbitrary vertex to $U'$ and contradict maximality. Similarly we know that, for every vertex $v\in U\setminus U'$, $v$ has at most $\log^{24}n$ neighbours outside of $U'\cup N_{G}(U')$, for otherwise $U'\cup \{v\}$ contradicts the maximality. Let $F$ be the set consisting of edges $uv$ with $u \in U\setminus U'$ and $v \in V(G)\setminus (U'\cup N_{G}(U'))$, so that $|F|\leq |U\setminus U'|\log^{24}n \leq s|U|$. Thus, we have, by the definition of expansion (and the choice of $F$), that 
\begin{equation}\label{eqn:rain}
\frac{\eps|U|}{\log^2n} \le |N_{G-F}(U)| \le |N_{G}(U')| \le (|U'|+1)\log^{24}n+1\leq 3|U'|\log^{24}n,
\end{equation}
so that $|U'|\geq \frac{\eps |U|}{3\log^{26}n}$, as required. Note that in the sequence of inequalities at \eqref{eqn:rain} we gain that $U'\neq \emptyset$ (as $|U|\geq 1$) in time to use this for the last inequality.
\end{proof}

We now show that we can ensure the conclusion of \Cref{lem:expandintorandom} holds for \emph{all} well-expanding sets simultaneously by taking a union bound, and then use Proposition~\ref{prop:well-expanding-core} to deduce an expansion property for all sets, as follows.

\begin{lemma}\label{lem:expandintorandom2} Suppose that $G$ is an $n$-vertex $(\eps,s)$-expander with $2^{-9}\leq \eps\leq 1$ and $s \ge 2\log^{24} n$. Let $V\subseteq V(G)$ be a random subset chosen by including each vertex independently at random with probability $\frac13$.

Then, with probability at least $1-o\left(1/n\right)$, for every $U\subseteq V(G)$ and every set $F\subseteq E(G)$ with $|F|\leq \frac{|U|}{\log^{27}n}$
\begin{equation}\label{eqn:B1UV}
|B_{G-F}^{\log^4n}(U,V)|> \frac{|V|}{2}.
\end{equation}
\end{lemma}
\begin{proof} 
Say a set $U'\subseteq V(G)$ \emph{expands well} in $G$ if 
$|N_{G}(U')| \ge |U'|\log^{24}{n}.$ Given a non-empty well-expanding set $U'\subseteq V(G)$ and a set of edges $F$ of size
at most $|U'|$, \Cref{lem:expandintorandom} applied to $U'$ implies that 
\begin{equation}\label{eq:1}
    |B_{G-F}^{\log^4n}(U', V)|> \frac{|V|}2
\end{equation}
fails with probability at most $e^{-\Omega(|U'|\log^2n)}.$

Now a union bound over all pairs $(U',F)$ such that $U'$ is a well-expanding set in $G$ and $F$ is a set of at most $|U'|$ edges tells us that \emph{some} such pair $(U',F)$ fails \eqref{eq:1} with probability at most 
\begin{align*}
    \sum_{(U',F)}  e^{-\Omega(|U'|\log^2n)} &\le 
    \sum_{u=1}^n\sum_{f=1}^{u} \binom{n}{u}\binom{n^2}{f} \cdot e^{-\Omega(u\log^2n)} \\
&    \le \sum_{u=1}^n u\cdot n^{3u}\cdot e^{-\Omega(u\log^2n)} \le \sum_{u=1}^n e^{-\Omega(u\log^2n)}=o(1/n).
\end{align*}

Thus, with probability $1-o(1/n)$, we can assume that \eqref{eq:1} holds for every well-expanding set $U'$ and set $F\subseteq E(G)$ with $|F|\leq |U'|$. We will now show that this implies \eqref{eqn:B1UV} holds for all $U\subseteq V(G)$ and $F\subseteq E(G)$ with $|F|\leq \frac{|U|}{\log^{27}n}$, completing the proof.

Let then $U\subseteq V(G)$ with $|U|\leq \frac23n$ and let $F\subseteq E(G)$ satisfy the (slightly weaker) condition $|F|\leq \frac{2 |U|}{\log^{27}n}$. Then, by \Cref{prop:well-expanding-core}, there is a set $U'\subseteq U$ which is well-expanding for which $|U'|\geq \frac{\eps|U|}{3\log^{26}n}$. Noting that $|F|\leq |U'|$ (as we may assume $n$ is large with probability $1-o(1/n)$), we therefore have that
\[
|B_{G-F}^{\log^4n}(U,V)|\geq |B_{G-F}^{\log^4n}(U',V)|> \frac{|V|}{2}.
\]
Finally, consider $U\subseteq V(G)$ with $|U|> \frac23n$ and let $F\subseteq E(G)$ satisfy $|F|\leq \frac{|U|}{\log^{27}n}$. Let $\bar{U}\subseteq U$ be an arbitrary subset with $\frac{n}2\leq |\bar{U}|\leq \frac23 n$, so that we have $|F|\leq \frac{2|\bar{U}|}{\log^{27}n}$, and hence, from what we have just shown,
\[
|B_{G-F}^{\log^4n}(U,V)|\geq |B_{G-F}^{\log^4n}(\bar{U},V)|> \frac{|V|}{2},
\]
as required.
\end{proof}


\subsection{Path connectedness through a random subset in expanders}\label{subsec:well-connectivity-weak-expander}

We are now ready to prove Theorem~\ref{thm:pathconnect}. As discussed at the start of this section, we first split the edges of the graph $G$ into expanders, before applying Lemma~\ref{lem:expandintorandom2} to each of these, to get (with high probability), a strong enough expansion condition to apply \Cref{lem:connectmanypairs}.

\begin{proof}[ of Theorem~\ref{thm:pathconnect}] To recap: we have an $n$-vertex $(\eps,s)$-expander, $G$, with $2^{-7}\leq \eps \leq 1$ and $s\geq \log^{135}n$, and a random subset $V\subseteq V(G)$ where each vertex is included independently at random with probability $\frac13$. To prove Theorem~\ref{thm:pathconnect}, we need to show that, with high probability, $G$ is $(4\log^5n,2^8\log^5n)$-path connected through $V$.

Let $k=2^{17}\log^{42} n$, so that $s\geq 2^{12}\eps^{-1}k^2\log^4n$, and let $s'=\frac{\sqrt{s\eps}}{ 8k \log n} \ge 2\log ^{24} n$. Using Lemma~\ref{lem:partitionedgesintoexpanders}, take edge disjoint graphs $G_1,\ldots,G_k$ such that $E(G)=\bigcup_{i\in [k]}E(G_i)$ and, for each $i\in [k]$, $G_i$ is an $\left(\frac{\eps}4,s'\right)$-expander. 

Then, by \Cref{lem:expandintorandom2} and a union bound over the $k$ graphs $G_i$, with high probability we can assume that, for each $i\in [k]$ and every $U\subseteq V(G_i)$ and $F\subseteq E(G_i)$ with $|F|\leq \frac{|U|}{\log^{27}n}$,
\[
|B^{\log^4n}_{G_i-F}(U,V)|> \frac{|V|}{2}.
\]

Now, let $U\subseteq V(G)$ and $F\subseteq E(G)$ with $|F|\leq 2^{17}|U|\log^{15}n$. As the graphs $G_i$, $i\in [k]$, are edge disjoint, there must be some $i\in [k]$ with $|F\cap E(G_i)|\leq \frac{2^{17}|U|\log^{15}n}k=\frac{|U|}{\log^{27}n}$, and therefore 
\[
|B^{\log^4n}_{G-F}(U,V)|\geq |B^{\log^4n}_{G_i-F}(U,V)|> \frac{|V|}{2}.
\]
We also know by Chernoff's inequality (\Cref{chernoff}) that with high probability $|V| \ge \frac n8 +1$, and thus, by \Cref{lem:connectmanypairs}, applied with $k=2^8\log^{5}n$ and $\ell=\log^4 n$ we conclude that $G$ is $(4\log^5n,2^8\log^5n)$-connected, as desired.
\end{proof}

\section{Cycle decompositions}\label{sec:logstarproof}
In this section, we will prove our main results, \Cref{thm:logstar,thm:decompexander}. Before doing this we need to put together a few final ingredients.
In Section~\ref{sec:connectinexpander}, we established a very robust connectivity property of expanders. In \Cref{sec:skeletons} we will show that in an expander one can find a subgraph with few edges and yet (effectively) the same connectivity property through a random subset $V$, with high probability. We will refer to this subgraph as a \emph{skeleton} of our graph, divide the vertex set into three as $V(G)=V_1\cup V_2\cup V_3$, and find three matching skeletons. In \Cref{sec:lovasz} we show that any graph can be decomposed into few paths in such a way that no vertex is used as an endvertex many times. In \Cref{subsec:decompexpander} we combine these results to decompose any expander into linearly many cycles and a few leftover edges. As outlined in Section~\ref{subsec:sketch}, we achieve this by setting aside the skeletons, then decomposing the remainder of the graph into three sets of paths and finally using the connection properties of the skeletons to join the endvertices of these sets of paths, where the sets of paths are matched to the skeletons to ensure this creates edge disjoint cycles. These cycles decompose all the edges in the graph which are not in the skeleton and since the skeleton is chosen to be sparse this gives us the result. 

The final ingredient in the proof of \Cref{thm:decompexander}, given in \Cref{sec:decomp-general}, is to decompose an arbitrary graph into expanders and a few leftover edges via \Cref{lem:decompexander}, to each of which we can apply our expander decomposition result. All that will remain, then, is to iteratively apply \Cref{thm:decompexander} in \Cref{subsec:iteration}, while removing some additional cycles, to deduce \Cref{thm:logstar}.

\subsection{Finding the skeletons}\label{sec:skeletons}

To find the skeletons, we will use \Cref{thm:pathconnect} to embed a sparse well-connected `template' graph (from \Cref{lem:template}) with its edges replaced by relatively short edge disjoint paths, and show that the image of this embedding has the properties we need of a skeleton, as follows.

\begin{lemma}\label{lem:sparseconnect}
Let $G$ be an $n$-vertex graph which is an $(\eps,s)$-expander with $2^{-7}\leq \eps \leq 1$ and $s\geq \log^{135}n$. Let $V\subseteq V(G)$ be chosen by including each vertex independently at random with probability $\frac13$. Then, with high probability, there is a subgraph of $G$ with at most $2^{9}n\log^{10}n$ edges which is $(\log^7 n,2)$-path connected through $V$.
\end{lemma}
\begin{proof} By Theorem~\ref{thm:pathconnect} applied to $G$ and $V$, $G$ is with high probability $(4\log^5 n,2^8\log^5 n)$-path connected through $V$. Note that, by Chernoff's inequality (\Cref{chernoff}), we can, in addition, ensure with high probability that $|V|\geq \frac n6$, and since our goal is to show a statement with high probability, we may assume that $n$ is large enough to apply \Cref{lem:template}. That is, by that lemma, we may assume there is an auxiliary graph $H$ with vertex set $V(H)=V(G)$ which is $\left(\frac14\log^2 n,2\right)$-path connected through $V$, and such that $\Delta(H) \le2^8\log^5 n$.

Note that $E(H)$ is a collection of pairs of vertices in $V(G)$ with the property that every vertex appears in at most $\Delta(H) \le 2^8\log^5 n$ pairs. Hence, since $G$ is $(4\log^5 n, 2^8\log^5 n)$-path connected through $V$, we can find for each $e\in E(H)$ a path $P_e$ through $V$ of length at most $4\log^5 n$, such that all the paths $P_e$, $e\in E(H)$, are edge disjoint. Let $G'$ have vertex set $V(G)$ and edge set $\bigcup_{e\in E(H)}E(P_e)$, noting that
\[
|E(G')|\leq 4\log^5n \cdot |E(H)|\leq 2n\log^5n \cdot \Delta(H)\leq 2n\log^5n\cdot 2^8\log^5n=2^{9}n\log^{10}n.
\]

Therefore, we need only show that $G'$ is $(\log^7 n,2)$-path connected. For this, let $\mathcal{P}\subseteq \binom{V(G)}{2}$ be a family of pairs of vertices from $V(G)$ with each vertex appearing in at most $2$ different pairs. Since $H$ is $\left(\frac14\log^2 n,2\right)$-path connected through $V$ we can find edge disjoint paths in $H$ through $V$, each of length at most $\frac14\log^2 n$, through $V$, joining each pair in $\mathcal{P}$. If we now replace each edge $e$ of $H$ used by one of these paths with $P_e$ we obtain a collection of edge disjoint \emph{walks} in $G'$ of length at most $\log^7 n$, through $V$, joining each pair in $\mathcal{P}$. Replacing each of the walks with its shortest subwalk joining the same endvertices, we obtain paths which are edge disjoint, each have length at most $\log^7n$, and which connect the vertex pairs in $\mathcal{P}$. Thus, $G'$ is $(\log^7 n,2)$-path connected, as claimed.
\end{proof}


\subsection{Lov\'asz path covering with well-spread endvertices}\label{sec:lovasz}
As quoted in the introduction, Lov\'asz proved the following classical decomposition result in 1968.

\begin{theorem}[Lov\'asz~\cite{lovasz1968covering}]\label{thm:lovasz}
Every $n$-vertex graph can be decomposed into at most $\frac n2$ paths and cycles.
\end{theorem}

Theorem~\ref{thm:lovasz} almost provides the path decompositions that we need. At the expense of using perhaps slightly more paths, we can ensure in addition that no vertex is used often as an endvertex of the paths in the decomposition through the following simple deduction.

\begin{corollary}\label{cor:lovasz}
Every $n$-vertex graph $G$ can be decomposed into paths so that each vertex of $G$ is an endvertex of at most two paths in the decomposition.
\end{corollary}

\begin{proof}
Form a graph $G'$ from $G$ by adding a new vertex $v_0$ and an edge from $v_0$ to each vertex $v\in V(G)$ for which $d_G(v)$ is even. By Theorem~\ref{thm:lovasz}, there is a collection $\mathcal{C}$ of at most $\frac{n+1}2$ cycles and paths which decomposes $G'$. Note that each vertex $v\in V(G)$ has odd degree in $G'$, and therefore must be an endvertex of an odd number of paths in $\mathcal{C}$, and thus, in particular, must be an endvertex of some path in $\mathcal{C}$. 
As the paths in $\mathcal{C}$ together have at most $n+1$ endvertices, each vertex in $G$ is the endvertex of at most 1 path in $\mathcal{C}$, as otherwise there would need to be $n-1+3 > n+1$ endvertices. Note that, furthermore, as there must be at least $n$ endvertices together for the paths in $\mathcal{C}$, we must have that $\mathcal{C}$ contains at least $\frac n2$ paths. Thus, as $|\mathcal{C}|\leq \frac{n+1}2$, $\mathcal{C}$ is in fact a collection consisting only of paths, with no cycles.

Now, for each path $P\in \mathcal{C}$, if $v_0\in V(P)$ then remove the vertex $v_0$ from $P$, and let $\mathcal{C'}$ be the collection of all the resulting paths. Then, $\mathcal{C}'$ is a decomposition of $G$ into paths. Furthermore, observe that each vertex $v\in V(G)$ is an endvertex of a path $P\in \mathcal{C}$ only if $v$ was an endvertex of some path in $\mathcal{C}$ which contained $P$ or if $P$ was created by removing the edge $vv_0$ from a path in $\mathcal{C}$. Thus, each vertex is an endvertex of at most two paths in $\mathcal{C'}$, so that $\mathcal{C}'$ decomposes $G$ as required. 
\end{proof}

\subsection{Decomposing expanders}\label{subsec:decompexpander}
In this section, we will decompose an expander into linearly many cycles and a few leftover edges, as follows.

\begin{lemma}\label{lem:decompexander} Any sufficiently large $n$-vertex graph $G$ which is an $(\eps,s)$-expander with $2^{-5}\leq \eps \leq 1$ and $s \ge \log^{273} n$ can be decomposed into at most $3n$ cycles and at most $2^{11}n\log^{10}n$ edges.
\end{lemma}
\begin{proof} Let $s'=\frac{\sqrt{s\eps}}{24\log n}\ge \log^{135} n$. Using Lemma~\ref{lem:partitionedgesintoexpanders} with $k=3$, take edge disjoint $\left(\frac{\eps}4,s'\right)$-expander subgraphs $G_1,G_2$ and $G_3$ of $G$, each with vertex set $V(G)$, so that $E(G)=E(G_1)\cup E(G_2)\cup E(G_3)$. Next, partition $V(G)=V_1\cup V_2\cup V_3$ by assigning each vertex to a set in the partition uniformly and independently at random. Using Lemma~\ref{lem:sparseconnect}, for each $i\in [3]$, find a subgraph $G_i'\subseteq G_i$ with at most $2^{9}n\log^{10}n$ edges which is $(\log^7 n,2)$-path connected through $V_i$. 

For each $i\in [3]$, let $H_i$ be the graph with vertex set $V(G)\setminus V_i$ whose edges are the edges of $G-G_1'-G_2'-G_3'$ lying within $V_{i+1}$ or between $V_{i+1}$ and $V_{i+2}$, with these indices taken appropriately modulo $3$, noting that these graphs partition $G-G_1'-G_2'-G_3'$. For each $i\in [3]$, then, apply Corollary~\ref{cor:lovasz} to decompose $H_i$ into a collection of paths $\mathcal{P}_i$ with the property that no vertex is an endvertex of more than two of the paths in $\mathcal{P}_i$, noting that, in particular, this implies that $|\mathcal{P}_i|\le n.$ Since $G_i'$ is $(\log^7 n,2)$-path connected through $V_i$ and the paths in $\mathcal{P}_i$ have no vertices in $V_i$, for each path $P\in\mathcal{P}_i$, we can find a path $Q_P$ through $V_i$ in $G_i'$ joining the endvertices of $P$, so that all these new paths are edge disjoint. As $V(H_i) \cap V_i =\emptyset$, for each $P\in \mathcal{P}_i$, $P\cup Q_P$ is a cycle, and furthermore, as $H_i$ and $G_i'$ are edge disjoint, the cycles form a collection, $\mathcal{C}_i$ say, of at most $n$ edge disjoint cycles whose edges contain all of the edges of $H_i$.

As the subgraphs $H_i\cup G_i'$, $i\in [3]$, are edge disjoint, $\mathcal{C}_1\cup \mathcal{C}_2\cup \mathcal{C}_3$ is a collection of at most $3n$ cycles which contains every edge of $G$ except for, perhaps, some edges in $G_1'\cup G_2'\cup G_3'$. Thus, these cycles cover all but at most $2^{11}n\log^{10} n$ edges, giving us the desired decomposition.
\end{proof}

\subsection{Decomposing a general graph}
\label{sec:decomp-general}
We are now ready to prove \Cref{thm:decompexander}, which we do in the following more quantitative form for convenience.
\begin{theorem}\label{thm:decompexander-explicit}
Any $n$-vertex graph can be decomposed into at most $6n$ cycles and $O(n\log^{274}n)$ edges.
\end{theorem}

\begin{proof} Let $n_0\geq 2^{12}$ be sufficiently large so that the statement of \Cref{lem:decompexander} holds for every graph with at least $n_0$ vertices, and let $C=3n_0$. Let $s=\log^{273}n$ and $\eps=2^{-5}$. Using Lemma~\ref{splitting-into-expanders}, decompose $G$ into subgraphs $G_1,\ldots, G_r$ and at most $4sn\log n$ edges so that $|G_1|+\ldots+|G_r| \le 2n$ and, for each $i\in [r]$, $G_i$ is an $(\eps,s)$-expander. Let $I\subseteq [r]$ be the set of $i\in [r]$ with $|G_i|\geq n_0$, so that 
\[
\big|E(G)\setminus\big(\bigcup_{i\in I}E(G_i)\big)\big|\leq n_0\cdot \sum_{i=1}^r|G_i|+4s n \log n\leq 2n_0s\cdot n \log n.
\]
For each $i\in I$, using Lemma~\ref{lem:decompexander}, decompose $G_i$ into at most $3|G_i|$ cycles and $2^{11}|G_i|\log^{10}n$ edges, giving in total at most $3|G_1|+\ldots+3|G_r| \le 6n$ cycles and at most $2^{11}(|G_1|+\ldots+|G_r|)\log^{10}n\leq n_0s\cdot n \log n$ edges. In combination with the edges of $G-\bigcup_{i\in I}G_i$ this gives a decomposition of $G$ into at most $6n$ cycles and $3n_0sn\log n=Cn\log^{274}n$ edges.
\end{proof}

\subsection{Long cycles in expanders}
A final ingredient we need before proving Theorem~\ref{thm:logstar} is to show that, after removing long cycles and adding them to a decomposition, any expander in the resulting graph must be quite small. That is to say, any expander contains a long cycle, even if it has no robustness at all. This follows from a result of Krivelevich \cite{michael2019expanders} although since we do not need the full power or generality of that result, for completeness, we include a short proof using the Depth First Search (DFS) algorithm (first used in this manner in \cite{dfs2012}), as follows.

\begin{lemma}\label{lem:onelongcycle}
Any $n$-vertex $(\eps,0)$-expander, with $\eps \ge 2^{-5}$ and $n\geq 2^{30}/\eps^2$, contains a cycle of length $\Omega\left(\frac{n}{\log^4 n}\right)$.
\end{lemma}
\begin{proof}
Let $G$ be our $(\eps,0)$-expander. Observe first that the expansion condition guarantees $G$ is connected since otherwise, we could find a connected component of $G$ of size less than $\frac{|G|}2$, whose vertex set then does not expand at all. 

We now run the DFS algorithm on $G$ as follows. At any point during the process, we have the set of unexplored vertices $U$, the path $P$ with active endvertex $t(P)$, and the set of processed vertices $R$. Picking an arbitrary vertex $r\in V(G)$, We start with $U=V(G)\setminus \{r\}, R=\emptyset$ and with $P$ being the path with vertex set $\{r\}$, and set $t(P)=r$. 
At each step, if there is a neighbour $v$ of $t(P)$ in $U$ we add it to $P$ with the edge $t(P)v$ and let $t(P)=v$. Otherwise, we move $t(P)$ from $P$ to $R$ and set its neighbour in $P$ as the new $t(P)$.

In the above process, in each step, we either move precisely one vertex from $U$ to $P$ or precisely one vertex from $P$ to $R$. Note also that at any point in the process, there are no edges between $U$ and $R$ since a vertex is only moved to $R$ once it has no neighbours in $U$, and $U$ only ever has vertices removed from it. 
Finally, as $G$ is connected, note that the process finishes with all the vertices being in $R$, and with $P$ and $U$ being empty.

Thus, we start with $|U|=n-1$ and $|R|=0$ and finish with $|U|=0$ and $|R|=n$, at each step reducing $|U|$ by one or increasing $|R|$ by one. Therefore, at some point in the process, we must have $|U|=|R|$. Since there are no edges between $U$ and $R$ we know that all the neighbours of $U$ must belong to $P$, which therefore has size $|P|=n-2|U|\geq |N_G(U)| \geq \frac{\eps |U|}{\log^2n}$, where the last inequality follows by the expansion property applied to $U$, which we can do since $|U|=\frac{n-|P|}2\leq \frac{n}{2}$. This in turn implies that $|P|\geq  \frac{\eps n}{3\log^2 n}$ since either $|U| \ge \frac{n}{3}$ or $|P|=n-2|U| \ge \frac n3$.

Now, let $X$, $Y$, $Z$ be sets of consecutive vertices of $P$ in that order which partition $V(P)$ so that $|X|,|Z|\geq \frac{|P|}3$ and $\frac{\eps^2 n}{18\log^4n}\leq |Y|< \frac{\eps^2 n}{9\log^4 n}$. 
If $X$ and $Z$ are connected by some path in $G\setminus Y$, then take a shortest path, $Q$ say, between $X$ and $Z$ in $G\setminus Y$ and note that, combined with the segment of $P$ between the endvertices of $Q$, this gives a cycle containing each vertex in $Y$, which thus has size $\Omega\left(\frac{n}{\log^4 n}\right)$. 
If $X$ and $Z$ are not connected by a path in $G\setminus Y$, then we can take a partition $V(G)\setminus Y=X'\cup Z'$ with no edges between $X'$ and $Z'$ in $G$, and $X\subseteq X'$ and $Z\subseteq Z'$. Without loss of generality, suppose that $|X'|\leq \frac n2$. By the expansion condition we have $|N_G(X')|\geq \frac{\eps |X|}{\log^2n} \geq \frac{\eps |P|}{3\log^2n} \geq \frac{\eps^2 n}{9\log^4 n}$, yet we also have $|N_G(X')|\leq |Y|< \frac{\eps^2 n}{9\log^4n}$, a contradiction.
\end{proof}

\subsection{Proof of \texorpdfstring{\Cref{thm:logstar}}{Theorem 2}}
\label{subsec:iteration}
To decompose a graph into cycles and edges and prove \Cref{thm:logstar}, we now repeatedly do the following:
\begin{itemize} 
\item letting $d$ be the average degree of the graph consisting of the remaining edges, we remove maximally many edge disjoint cycles with length at least $d$, 
\item we then decompose the remaining edges exactly into expanders (using \Cref{splitting-into-expanders} with $s=0$) and show that these must be small subgraphs as they each have no cycle with length at least $d$ (using \Cref{lem:onelongcycle}),
\item and finally we decompose each of these small subgraphs into cycles and edges using \Cref{thm:decompexander-explicit}.
\end{itemize}
Together this comprises the iterative step we use, which decomposes the $n$-vertex graph $G$ with average degree $d$ into $O(n)$ cycles and $n\log^{O(1)}d$ edges. We then iterate on the graph of the edges in this decomposition, noting that its average degree is much smaller than $d$. We state and prove the outcome of one iterative step as the following lemma, for convenience and its own interest.

\begin{lemma}\label{lem:density}
Any $n$-vertex graph $G$ with average degree $d\ge 2$ can be decomposed into $O(n)$ cycles and a subgraph with average degree $O(\log^{274} d)$.
\end{lemma}

\begin{proof} Let $\mathcal{C}$ be a maximal collection of edge disjoint cycles with length at least $d$ in $G$. As there are $\frac{nd}2$ edges in $G$, we have $|\mathcal{C}|\leq \frac n2$. Let $G'$ be $G$ with the edges of the cycles in $\mathcal{C}$ removed, so that $G'$ has no cycles with length at least $d$.
Apply \Cref{splitting-into-expanders} with $s=0$ and  $\eps=2^{-5}$ to obtain a full decomposition (since $s=0$) of $G'$ into subgraphs $G_1,\ldots, G_k$, such that $|G_1|+\ldots+|G_k|\le 2n$ and, for each $i\in [k]$, $G_i$ is a $(2^{-5},0)$-expander. For each $i\in [k]$, as $G'$ and hence $G_i$ has no cycle with length at least $d$, we have by \Cref{lem:onelongcycle} that $|G_i| =  O(d \log^4 d)$. Using \Cref{thm:decompexander-explicit}, we decompose each $G_i$ into at most $6|G_i|$ cycles and $O(|G_i| \log^{274} |G_i|)= O(|G_i| \log^{274} d)$ edges. Collecting these cycles and edges over all $i\in [k]$, and including the cycles in $\mathcal{C}$, we get a decomposition of $G$ into at most
$$\frac n2+\sum_{i=1}^k 6|G_i| \le 13 n$$ cycles and $$O\left(\sum_{i=1}^k |G_i| \log^{274} d\right)=O(n \log^{274} d)$$ edges. Noting these edges form a subgraph with average degree $O(\log^{274} d)$ completes the proof. 
\end{proof}

Finally, by iterating \Cref{lem:density}, we can prove \Cref{thm:logstar}, i.e., that any $n$-vertex graph can be decomposed into $O(n\log^{\star}n)$ cycles and edges.


\begin{proof}[ of \Cref{thm:logstar}]
Using Lemma~\ref{lem:density}, let $C\geq 1$ be large enough that any $n$-vertex graph with average degree $d\geq 2$ has a decomposition into at most $Cn$ cycles and a subgraph with average degree at most $C\log^{274}n$. Let $G_0$ be any $n$-vertex graph, and, for each $i\geq 0$, let $G_{i+1}$ be a graph with the fewest edges that can be formed by removing at most $Cn$ edge disjoint cycles from $G_{i}$. For each $i\geq 0$, let $d_i$ be the average degree of $G_i$, so that we have $d_{i+1}\leq C\log^{274}d_i$ for each $i\geq 0$ for which $d_i\geq 2$.

Let $\ell$ be the largest integer such that the log function applied iteratively $\ell$ times to $n$ is still above $300C$, i.e., the largest integer such that $\log^{[\ell]}n\geq 300C$. Note that $\ell\leq \log^{\star}n$ and $\log^{[\ell]}n< 2^{300C}$. We will show, for each $0 \le i\le \ell$, that $d_i\leq C(300\log^{[i]}n)^{274}$. Note that $d_0\leq n$, so that this easily holds with $i=0$. Then, assuming it is true for some $0 \le i \le \ell-1$ and that $d_i\geq 2$ (for otherwise $d_{i-1}\leq d_i\leq 2$), we have
\begin{align}
d_{i+1}&\leq C\log^{274}d_i\leq C(\log (C(300\log^{[i]}n)^{274}))^{274}\nonumber\\&= C(\log C+274\log300+274\log^{[i+1]}n)^{274}
\leq C(300\log^{[i+1]}n)^{274},\label{eqn:Clogi}
\end{align}
where in the last inequality we used $26\log^{[i+1]}n \ge 26\log^{[\ell]} n\ge 26\cdot 300C \ge \log C + 274 \log 300$. 

Thus, in particular, we have that the average degree of $G_\ell$ is at most $C(300\log^{[\ell]}n)^{274}\leq C(300\cdot 2^{300C})^{274}=O(1)$. Furthermore, to get from $G_0$ to $G_\ell$, we have removed at most $C\ell\leq Cn\log^{\star}n$ cycles. Therefore, $G_0$ has a decomposition into $O(n\log^{\star}n)$ cycles and $O(n)$ edges, which, as $G_0$ is an arbitrary $n$-vertex graph, completes the proof.\end{proof}


\section{Concluding remarks}\label{sec:final}
\textbf{Our results.} In this paper, we gave new bounds on two of the most central open problems on cycle decompositions -- the Erd\H{o}s-Gallai conjecture from 1966 that any $n$-vertex graph can be decomposed into $O(n)$ cycles and edges, and Haj\'os's conjecture from 1968 asserting that any $n$-vertex Eulerian graph can be decomposed into at most $\frac{n}{2}$ cycles, where the bound in Haj\'os's conjecture follows easily from Theorem~\ref{thm:logstar}. Indeed, given any $n$-vertex Eulerian graph, we can apply this result first to remove $O(n\log^{\star}n)$ cycles and leave only $O(n \log^{\star} n)$ remaining edges. The remaining edges then still form an Eulerian graph, which has a cycle decomposition by the observation of Veblen quoted in the introduction. As this cycle decomposition has fewer cycles than the number of these edges, we get, altogether, a decomposition of the original Eulerian graph into $O(n\log^{\star}n)$ cycles.

\textbf{Lower bounds for the Erd\H{o}s-Gallai conjecture.} As noted in the introduction, Haj\'os's conjecture implies the Erd\H{o}s-Gallai conjecture, as any $n$-vertex graph can be decomposed into an Eulerian graph and at most $n-1$ edges --- for example by taking the union of a maximal collection of disjoint cycles, and the edges in the remaining acyclic subgraph. Therefore, if Haj\'os's conjecture holds, then any $n$-vertex graph would have a decomposition into at most $\frac{n}{2}$ cycles and at most $n-1$ edges, and thus at most $\frac32n$ cycles and edges. It is known that only $(\frac{3}{2}+o(1))n$ cycles and edges are needed to decompose any $n$-vertex graph with linear minimum degree (due to Gir{\~a}o, Granet, K{\"u}hn, and Osthus~\cite{girao2021path}), and also that $\frac{3}{2}$ is best possible here. This latter fact was observed by Erd\H{o}s~\cite{erdHos1983some} in 1983 who remarked that there are graphs requiring $(\frac{3}{2}-o(1))n$ cycles and edges, likely referring to the following generalisation of an example of Gallai (see \cite{erdos1966representation}).  Take $k\in \N$ and consider the complete bipartite graph $G$ with disjoint vertex classes $A$ and $B$ with $|A|=2k+1$ and $|B|=n-2k-1$. Each vertex in $B$ has odd degree in $G$, and $G[B]$ contains no edges, so any decomposition of $G$ into cycles and edges must contain at least $|B|$ edges. As each cycle in $G$ has length at most $2|A|$, there is thus no decomposition of $G$ into fewer than $$|B|+\frac{|A||B|-|B|}{2|A|}=\left(\frac{3}{2}-\frac{1}{2|A|}\right)|B|=\left(\frac{3}{2}-\frac{1}{4k+2}-o(1)\right)n$$ cycles and edges.

\textbf{Our tools.} Many of our tools decompose an $n$-vertex graph into $O(n)$ cycles and a `leftover' subgraph, and might prove useful towards settling the conjecture in full. For example, the proof of \Cref{lem:density} shows that any $n$-vertex graph with no cycles longer than $t$ decomposes into $O(n)$ cycles and $O(n\log^{274}t)$ edges. Furthermore, for any constant $k$, running our iteration argument using \Cref{lem:density} $k+1$ times shows that any $n$-vertex graph can be decomposed into $O(kn)$ cycles and $O(n\log^{[k]}n)$ edges (see \eqref{eqn:Clogi}). This latter result reduces the Erd\H{o}s-Gallai conjecture to the case of arbitrarily sparse graphs, although this seems only to focus on the most difficult case. The main bottleneck in our argument seems to be the number of edges left uncovered in the almost decomposition into robust expanders, i.e.\ when applying \Cref{splitting-into-expanders}. It appears hard to reduce the number of edges enough to make an improvement on \Cref{thm:logstar} while getting enough properties in the expansion (robust or otherwise) to aid any cycle decomposition.

\textbf{Potential further applications.} 
In addition to their application towards the Erd\H{o}s-Gallai conjecture, we believe the tools developed here could be useful in other settings due to the variety of applications that have been found for sublinear expansion since its introduction in its original form by Koml\'os and Szemer\'edi, as well as the only very recent use of the robust sublinear expansion discussed in \Cref{sec:generalexpansionchat}. In fact, one can use some of our ideas, specifically a more precise version of \Cref{prop:expansion-red-blue}, to give a simpler proof of a result of Tomon \cite{tomon2022robust} on finding rainbow cycles in properly coloured graphs, and give a different proof on a result of Wang \cite{wang2022rainbow} (itself an improvement of a result of Tomon~\cite{tomon2022robust}). In addition, Letzter~\cite{shobro} has very recently adapted some of our methods in order to find separating paths systems.

\medskip

\textbf{Acknowledgments.} We would like to thank Benny Sudakov, Jacob Fox, Stefan Glock and Shoham Letzter for helpful comments which improved this paper. The first author would like to gratefully acknowledge the support of the Oswald Veblen Fund.


\begin{thebibliography}{10}

\bibitem{aharoni2000hall}
R.~Aharoni and P.~Haxell.
\newblock Hall's theorem for hypergraphs.
\newblock {\em J. Graph Theory}, 35(2):83--88, 2000.

\bibitem{alon-spencer}
N.~Alon and J.~H. Spencer.
\newblock {\em The {P}robabilistic {M}ethod}.
\newblock John Wiley \& Sons, 4th edition, 2016.

\bibitem{network_decompositions1989}
B.~Awerbuch, M.~Luby, A.~Goldberg, and S.~Plotkin.
\newblock Network decomposition and locality in distributed computation.
\newblock In {\em 30th Annual Symposium on Foundations of Computer Science},
  pages 364--369, 1989.

\bibitem{dfs2012}
I.~Ben-Eliezer, M.~Krivelevich, and B.~Sudakov.
\newblock Long cycles in subgraphs of (pseudo)random directed graphs.
\newblock {\em J. Graph Theory}, 70(3):284--296, 2012.

\bibitem{blanche2021gallai}
A.~Blanch{\'e}, M.~Bonamy, and N.~Bonichon.
\newblock Gallai’s path decomposition for planar graphs.
\newblock In {\em Extended Abstracts EuroComb 2021: European Conference on
  Combinatorics, Graph Theory and Applications}, pages 758--764. Springer,
  2021.

\bibitem{bondyEG1990}
J.~A. Bondy.
\newblock Small cycle double covers of graphs.
\newblock In {\em Cycles and rays ({M}ontreal, {PQ}, 1987)}, volume 301 of {\em
  NATO Adv. Sci. Inst. Ser. C: Math. Phys. Sci.}, pages 21--40. Kluwer Acad.
  Publ., Dordrecht, 1990.

\bibitem{broder1996efficient}
A.~Z. Broder, A.~M. Frieze, S.~Suen, and E.~Upfal.
\newblock An efficient algorithm for the vertex-disjoint paths problem in
  random graphs.
\newblock In {\em Proceedings of the seventh annual ACM-SIAM symposium on
  Discrete algorithms}, pages 261--268, 1996.

\bibitem{chakraborti2021well}
D.~Chakraborti, J.~Kim, J.~Kim, M.~Kim, and H.~Liu.
\newblock Well-mixing vertices and almost expanders.
\newblock {\em Proc. Amer. Math. Soc.}, 150(12):5097--5110, 2022.

\bibitem{chung1980}
F.~R.~K. Chung.
\newblock On the coverings of graphs.
\newblock {\em Discrete Math.}, 30(2):89--93, 1980.

\bibitem{designs_applications1989}
C.~J. Colbourn and P.~C. van Oorschot.
\newblock Applications of combinatorial designs in computer science.
\newblock {\em ACM Comput. Surv.}, 21(2):223–250, jun 1989.

\bibitem{conlon2014cycle}
D.~Conlon, J.~Fox, and B.~Sudakov.
\newblock Cycle packing.
\newblock {\em Random Struct. Algorithms}, 45(4):608--626, 2014.

\bibitem{dean1986smallest}
N.~Dean.
\newblock What is the smallest number of dicycles in a dicycle decomposition of
  an {E}ulerian digraph?
\newblock {\em J. Graph Theory}, 10(3):299--308, 1986.

\bibitem{dean2000gallai}
N.~Dean and M.~Kouider.
\newblock Gallai's conjecture for disconnected graphs.
\newblock {\em Discrete Math.}, 213(1-3):43--54, 2000.

\bibitem{erdHossome1971}
P.~Erd\H{o}s.
\newblock Some unsolved problems in graph theory and combinatorial analysis.
\newblock In {\em Combinatorial {M}athematics and its {A}pplications ({P}roc.
  {C}onf., {O}xford, 1969)}, pages 97--109. Academic Press, London, 1971.

\bibitem{erdos1973problems}
P.~Erd\H{o}s.
\newblock Problems and results in combinatorial analysis.
\newblock In {\em Colloq. Internat. Theor. Combin. Rome}, pages 3--17, 1973.

\bibitem{erdossolved1981}
P.~Erd\H{o}s.
\newblock On the combinatorial problems which {I} would most like to see
  solved.
\newblock {\em Combinatorica}, 1(1):25--42, 1981.

\bibitem{erdHos1983some}
P.~Erd{\H{o}}s.
\newblock On some of my conjectures in number theory and combinatorics.
\newblock In {\em Proceedings of the fourteenth Southeastern conference on
  combinatorics, graph theory and computing}, volume~39, pages 3--19, 1983.

\bibitem{erdos1966representation}
P.~Erd{\H{o}}s, A.~W. Goodman, and L.~P{\'o}sa.
\newblock The representation of a graph by set intersections.
\newblock {\em Can. J. Math.}, 18:106--112, 1966.

\bibitem{fan2002subgraph}
G.~Fan.
\newblock Subgraph coverings and edge switchings.
\newblock {\em J. Comb. Theory. Ser. B}, 84(1):54--83, 2002.

\bibitem{fan2003covers}
G.~Fan.
\newblock Covers of {E}ulerian graphs.
\newblock {\em J. Comb. Theory. Ser. B}, 89(2):173--187, 2003.

\bibitem{fernandez2022build}
I.~G. Fern{\'a}ndez and H.~Liu.
\newblock How to build a pillar: a proof of {T}homassen's conjecture.
\newblock {\em J. Comb. Theory. Ser. B}, 162:13--33, 2023.

\bibitem{fernandez2022nested}
I.~Gil~Fern{\'a}ndez, J.~Kim, Y.~Kim, and H.~Liu.
\newblock Nested cycles with no geometric crossings.
\newblock {\em Proc. Amer. Math. Soc. Ser. B}, 9(03):22--32, 2022.

\bibitem{girao2021path}
A.~Gir{\~a}o, B.~Granet, D.~K{\"u}hn, and D.~Osthus.
\newblock Path and cycle decompositions of dense graphs.
\newblock {\em J. London Math. Soc.}, 2021.

\bibitem{glebov2013hamilton}
R.~Glebov.
\newblock {\em On {H}amilton cycles and other spanning structures}.
\newblock PhD thesis, Freie Universit{\"{a}}t Berlin, 2013.

\bibitem{glock2016optimal}
S.~Glock, D.~K{\"u}hn, and D.~Osthus.
\newblock Optimal path and cycle decompositions of dense quasirandom graphs.
\newblock {\em J. Comb. Theory. Ser. B}, 118:88--108, 2016.

\bibitem{haslegrave2021crux}
J.~Haslegrave, J.~Hu, J.~Kim, H.~Liu, B.~Luan, and G.~Wang.
\newblock Crux and long cycles in graphs.
\newblock {\em SIAM J. Discrete Math}, 36(4):2942--2958, 2022.

\bibitem{haslegrave2021ramsey}
J.~Haslegrave, J.~Hyde, J.~Kim, and H.~Liu.
\newblock Ramsey numbers of cycles versus general graphs.
\newblock {\em Forum Math. Sigma}, 11:e10, 2023.

\bibitem{haslegrave2021extremal}
J.~Haslegrave, J.~Kim, and H.~Liu.
\newblock {Extremal density for sparse minors and subdivisions}.
\newblock {\em Int. Math. Res. Not.}, 2021.

\bibitem{haxell2001tree}
P.~E. Haxell.
\newblock Tree embeddings.
\newblock {\em J. Graph Theory}, 36(3):121--130, 2001.

\bibitem{expander-survey}
S.~Hoory, N.~Linial, and A.~Wigderson.
\newblock Expander graphs and their applications.
\newblock {\em Bull. Amer. Math. Soc. (N.S.)}, 43(4):439--561, 2006.

\bibitem{jiang2021rainbow}
T.~Jiang, A.~Methuku, and L.~Yepremyan.
\newblock Rainbow {T}ur{\'a}n number of clique subdivisions.
\newblock {\em Eur. J. Comb.}, 110:103675, 2023.

\bibitem{jukna_complexity2006}
S.~Jukna.
\newblock On graph complexity.
\newblock {\em Combin. Probab. Comput.}, 15(6):855--876, 2006.

\bibitem{kim2017komlos}
J.~Kim, H.~Liu, M.~Sharifzadeh, and K.~Staden.
\newblock Proof of {K}oml\'{o}s's conjecture on {H}amiltonian subsets.
\newblock {\em Proc. Lond. Math. Soc. (3)}, 115(5):974--1013, 2017.

\bibitem{K-Sz-1}
J.~Koml{\'o}s and E.~Szemer{\'e}di.
\newblock Topological cliques in graphs.
\newblock {\em Combin. Probab. Comput.}, 3(2):247--256, 1994.

\bibitem{K-Sz-2}
J.~Koml{\'o}s and E.~Szemer{\'e}di.
\newblock Topological cliques in graphs {II}.
\newblock {\em Combin. Probab. Comput.}, 5(1):79--90, 1996.

\bibitem{dano2015}
D.~Kor\'{a}ndi, M.~Krivelevich, and B.~Sudakov.
\newblock Decomposing random graphs into few cycles and edges.
\newblock {\em Combin. Probab. Comput.}, 24(6):857--872, 2015.

\bibitem{michael2019expanders}
M.~Krivelevich.
\newblock Long cycles in locally expanding graphs, with applications.
\newblock {\em Combinatorica}, 39(1):135--151, 2019.

\bibitem{shobro}
S.~Letzter.
\newblock Separating paths systems of almost linear size.
\newblock {\em preprint arXiv:2211.07732}, 2022.

\bibitem{letzter2023immersion}
S.~Letzter and A.~Gir{\~a}o.
\newblock Immersion of complete digraphs in {E}ulerian digraphs.
\newblock {\em Israel Journal of Mathematics}, to appear.

\bibitem{liu2017mader}
H.~Liu and R.~Montgomery.
\newblock A proof of {M}ader's conjecture on large clique subdivisions in
  {$C_4$}-free graphs.
\newblock {\em J. London Math. Soc. (2)}, 95(1):203--222, 2017.

\bibitem{liu2020solution}
H.~Liu and R.~Montgomery.
\newblock A solution to {E}rd{\H{o}}s and {H}ajnal's odd cycle problem.
\newblock {\em J. Amer. Math. Soc.}, 36(4):1191--1234, 2023.

\bibitem{liu2020clique}
H.~Liu, G.~Wang, and D.~Yang.
\newblock Clique immersion in graphs without a fixed bipartite graph.
\newblock {\em J. Comb. Theory. Ser. B}, 157:346--365, 2022.

\bibitem{lovasz1968covering}
L.~Lov{\'a}sz.
\newblock On covering of graphs.
\newblock In {\em Theory of Graphs (Proc. Colloq., Tihany, 1966)}, pages
  231--236. Academic Press New York, 1968.

\bibitem{lucas1883recreations}
{\'E}.~Lucas.
\newblock {\em R{\'e}cr{\'e}ations math{\'e}matiques}, volume~2.
\newblock Gauthier-Villars, 1883.

\bibitem{montgomery2015logarithmically}
R.~Montgomery.
\newblock Logarithmically small minors and topological minors.
\newblock {\em J. London Math. Soc.}, 91(1):71--88, 2015.

\bibitem{montgomery2021spanning}
R.~Montgomery.
\newblock Spanning cycles in random directed graphs.
\newblock {\em preprint arXiv:2103.06751}, 2021.

\bibitem{pyber1985erdHos}
L.~Pyber.
\newblock An {E}rd{\H{o}}s—{G}allai conjecture.
\newblock {\em Combinatorica}, 5(1):67--79, 1985.

\bibitem{pyberEG1991}
L.~Pyber.
\newblock Covering the edges of a graph by {$\ldots$}.
\newblock In {\em Sets, graphs and numbers ({B}udapest, 1991)}, volume~60 of
  {\em Colloq. Math. Soc. J\'{a}nos Bolyai}, pages 583--610. North-Holland,
  Amsterdam, 1992.

\bibitem{pyber1996covering}
L.~Pyber.
\newblock Covering the edges of a connected graph by paths.
\newblock {\em J. Comb. Theory. Ser. B}, 66(1):152--159, 1996.

\bibitem{shapira2015small}
A.~Shapira and B.~Sudakov.
\newblock Small complete minors above the extremal edge density.
\newblock {\em Combinatorica}, 35(1):75--94, 2015.

\bibitem{sudakov2022extremal}
B.~Sudakov and I.~Tomon.
\newblock The extremal number of tight cycles.
\newblock {\em Int. Math. Res. Not.}, 2022(13):9663--9684, 2022.

\bibitem{tomon2022robust}
I.~Tomon.
\newblock Robust (rainbow) subdivisions and simplicial cycles.
\newblock {\em preprint arXiv:2201.12309}, 2022.

\bibitem{veblen1}
O.~Veblen.
\newblock An application of modular equations in analysis situs.
\newblock {\em Ann. of Math. (2)}, 14(1-4):86--94, 1912/13.

\bibitem{veblen2}
O.~Veblen.
\newblock The {C}ambridge {C}olloquium, 1916, part ii. analysis situs.
\newblock {\em Bull. Amer. Math. Soc}, 30:357--358, 1924.

\bibitem{wang2022rainbow}
Y.~Wang.
\newblock Rainbow clique subdivisions.
\newblock {\em preprint arXiv:2204.08804}, 2022.

\bibitem{yan1998path}
L.~Yan.
\newblock {\em On path decompositions of graphs}.
\newblock PhD thesis, Arizona State University, 1998.

\end{thebibliography}

\end{document}